\documentclass[11pt]{article}

\title{An Immersed $C^0$ Interior Penalty Method for Biharmonic Interface Problems}
\author{Yuan Chen\thanks{Department of Mathematics, The Ohio State University, Columbus, OH 43210, USA. Email: chen.11050@buckeyemail.osu.edu},~~Xu Zhang\thanks{Department of Mathematics, Oklahoma State University, Stillwater, OK 74078, USA. Email: xzhang@okstate.edu. This author is partially supported by National Science Foundation grant DMS-2110833 }}
\date{}
\usepackage{graphicx,amsmath,amsthm,amsfonts,latexsym,amssymb,hyperref,comment,subeqnarray,mathrsfs}
\usepackage{hyperref,mathrsfs}
\usepackage{mathtools}
\usepackage[mathscr]{euscript}
\usepackage{tabularx}
\usepackage{footmisc}
\usepackage{titling}
\thanksmarkseries{arabic}
\usepackage[misc]{ifsym} 
\usepackage{booktabs}
\usepackage{subfigure}
\usepackage{caption}
\usepackage{geometry}
\usepackage{multirow}
\numberwithin{equation}{section}
\newtheorem{theorem}{Theorem}[section]

\newtheorem{lemma}{Lemma}[section]

\newtheorem{remark}{Remark}[section]
\usepackage{todonotes}
\usepackage{hyperref}
\usepackage{xcolor}
\usepackage{tabularx}
\usepackage[ruled,vlined,linesnumbered]{algorithm2e}
\usepackage{algpseudocode}
\hypersetup{colorlinks,linkcolor=[RGB]{0,103,165},citecolor=[RGB]{180,0,0}} 
\newcommand{\vertiii}[1]{{\left\vert\kern-0.25ex\left\vert\kern-0.25ex\left\vert #1 
    \right\vert\kern-0.25ex\right\vert\kern-0.25ex\right\vert}}
\newcommand{\aver}[1]{\left\{\!\!\left\{#1\right\}\!\!\right\}}

\newcommand{\jump}[1]{\left[\!\left[#1\right]\!\right]}

\newcommand{\mynorm}[1]{\left\|#1\right\|}

\definecolor{red2}{RGB}{204,0,0}
\definecolor{blue2}{RGB}{0,103,165}
\usepackage{chngcntr}
\counterwithin{figure}{section}
\SetKwInput{KwInput}{Input}
\usepackage{enumitem}
\setlist[description]{leftmargin=\parindent,labelindent=\parindent}
\newtheorem{case}{Case}

\newcommand{\YCrev}[1]{{\color{blue} #1}}

\begin{document}
\maketitle
\vspace{-13mm}
\begin{abstract}

In this paper, we introduce an immersed $C^0$ interior penalty method for solving two-dimensional biharmonic interface problems on unfitted meshes. To accommodate the biharmonic interface conditions, high-order immersed finite element (IFE) spaces are constructed in the least-squares sense. We establish key properties of these spaces including unisolvency and partition of unity, and verify their optimal approximation capability. These spaces are further incorporated into a modified $C^0$ interior penalty scheme with additional penalty terms on interface segments. The well-posedness of the discrete solution is proved. Numerical experiments with various interface geometries confirm optimal convergence of the proposed method in  $L^2$, $H^1$ and $H^2$ norms.\\

\end{abstract}

\section{Introduction}

The biharmonic equation arises from plate bending theory in continuum mechanics. Its numerical solution poses significant challenges, particularly when using conforming finite element methods, since constructing $C^1$-conforming  elements \cite{1968ArgyrisFriedScharpf} is notoriously difficult and expensive. To address this difficulty, a wide range of alternative discretization techniques have been developed over the past several decades. 

One direction involves nonconforming finite element methods \cite{1968Morley}, which relax inter-element continuity requirements. However, these elements lack a systematic hierarchy to higher-order spaces. Another approach is mixed finite element methods \cite{1985ArnoldBrezzi, 1978Falk}, which reformulate the fourth-order PDE into a system of second-order equations. While effective in some cases, the mixed formulation introduces saddle-point structures and requires stable element pairs that satisfy the Ladyshenskaya–Babu$\check{\text{s}}$hka–Brezzi (LBB) condition \cite{brezzi1974existence}, which is often nontrivial to construct. In addition, discontinuous Galerkin (DG) methods \cite{cheng2008discontinuous, 2007MozolevskiSuliBosing}, weak Galerkin (WG) methods \cite{2014MuWangYe} are very flexible on approximation spaces and mesh types, but they usually have a large number of degrees of freedom, which makes these methods costly to use. 

A balanced alternative between $H^2$-conforming finite elements and fully discontinuous Galerkin methods is the $H^1$ semi-conforming approximation, known as the $C^0$ interior penalty method \cite{2005BrennerSung, engel2002continuous}. This method is particularly attractive because it uses only standard $C^0$ Lagrange elements while preserving the positive-definiteness of the resulting linear system and maintaining a natural hierarchy to higher order. Its flexibility has enabled extensions to a variety of PDEs, including optimal control problems \cite{brenner2019c0}, phase field crystal equations \cite{diegel2020c0}, and Hamilton–Jacobi–Bellman equations \cite{brenner2021adaptive}, among others. 

In this paper, we focus on biharmonic interface problems, governed by fourth-order elliptic equations with discontinuous coefficients across heterogeneous media. Such problems arise in composite plate bending, multiphase materials, and phase transition models. Let $\Omega\subset \mathbb{R}^2$ be an open bounded domain separated by a closed interface $\Gamma$ such that $\overline{\Omega} = \overline{\Omega^+\cup\Omega^-\cup\Gamma}$. Consider the two-dimensional biharmonic problems: 
\begin{subequations}\label{eq: bihar}
	\label{equ:interface}
	\begin{eqnarray}
	\Delta (\beta(\mathbf{x}) \Delta u)  = ~f,&&\text{in } \Omega/\Gamma, \label{equ:bihm}\\
		u ~ = ~ 0, &&\text{on } \partial\Omega, \label{equ:bc1}\\
		\partial_n u  = ~ 0, &&\text{on } \partial\Omega, \label{equ:bc2}
	\end{eqnarray}
where the Laplacian is $\Delta:=\frac{\partial^2}{\partial x_1^2} + \frac{\partial^2}{\partial x_2^2}$.  $n = (n_1,n_2)$ and $t = (t_1,t_2)$ denote the outward normal and tangential unit vectors, respectively, and $\beta$ is a piecewise constant coefficient that takes distinct values $\beta^\pm$ on $\Omega^\pm$. Across the interface $\Gamma$, the solution satisfies continuity conditions on the displacement and the normal derivative (see \eqref{equ:jump1}-\eqref{equ:jump2}), as well as two Neumann-type interface conditions (see \eqref{equ:jump3}–\eqref{equ:jump4}):
\begin{eqnarray}
\jump{u}_{\Gamma}  ~ = ~0,&&\text{on } \Gamma, \label{equ:jump1} \\
		\jump{\partial_n u}_{\Gamma}  ~ = ~0,&&\text{on } \Gamma, \label{equ:jump2}\\
		\jump{\beta\partial_{nn} u }_{\Gamma}  ~ = ~0,&&\text{on } \Gamma, \label{equ:jump3}\\
		\jump{\partial_n \left( \beta\Delta u+\beta\partial_{tt} u\right)}_{\Gamma}  ~ = ~0,&&\text{on } \Gamma, \label{equ:jump4}
\end{eqnarray}
\end{subequations}
where the jump operator $\jump{v}_{\Gamma}:=v^+|_{\Gamma}-v^-|_{\Gamma}$. 
Here and thereafter, we denote
\[
\partial_n v = \sum_{i=1}^2 n_i\partial_iv,~~~~
\partial_t v = \sum_{i=1}^2 t_i\partial_iv,~~~~
\partial_{nn} v = \sum_{i=1}^2\sum_{j=1}^2 n_in_j\partial_{ij}v,~~~~
\partial_{tn} v  =\sum_{i=1}^2\sum_{j=1}^2 t_in_j\partial_{ij}v,
\]
where $\partial_i = \frac{\partial}{\partial{x_i}}$, and $\partial_{ij} = \frac{\partial^2}{\partial{x_i}\partial{x_j}}$.  
For biharmonic interface problems, while the Dirichlet interface conditions \eqref{equ:jump1}-\eqref{equ:jump2} are universal, there are two types of  Neumann interface conditions in the literature. The first one, derived from Stokes formulation  {\cite{ciarlet2002finite}}, takes the form
\begin{equation}
\jump{\beta\Delta u}_\Gamma = 0,~~~~\jump{\partial_n(\beta\Delta u)}_\Gamma = 0. 
\end{equation}
The second type is derived from the energy functional of the free thin plate bending model  {\cite{gander2012chladni}}, which takes the form of \eqref{equ:jump3}-\eqref{equ:jump4}.  {These two types of interface conditions arise from distinct modeling considerations, corresponding to different choices of energy functionals \cite{gander2017definition}. In this work, we consider the \emph{second type} interface condition. Extending the proposed method to the first type is, in principle, straightforward. 
}

 {Several approaches have been developed for biharmonic interface problems. Within the immersed finite element framework, mixed formulations have been explored in \cite{MR4450434,lin2011immersed}, where piecewise linear IFE functions are employed to approximate the resulting saddle-point problem. Motivated by this approach, a finite difference method with augmented variable was proposed in \cite{li2017augmented}. An immersed Morley element was constructed in \cite{cai2021nitsche}, where a Nitsche-type extended finite element method is used to weakly enforce interface conditions. CutFEM approaches  \cite{2020BurmanHansboLarson, 2022BurmanHansboLarson,2025BurmanHansboLarsonZahedi} use $C^1$ finite element spaces with Nitsche’s method for weak enforcement of the interface conditions. A reconstructed discontinuous Galerkin method on unfitted meshes was introduced in \cite{2023ChenLiLiu}, where the approximation space is constructed via patch reconstruction and interface conditions are enforced weakly through Nitsche-type techniques. More recently, a nonconforming virtual element method on unfitted polygonal meshes was developed in \cite{2025MaChenWang}. In addition, mixed formulations have been extended to biharmonic interface problems posed on surfaces in \cite{2024ChenXiaoFeng}.}

 {As a numerical framework for PDE interface problems}, the main idea of  {IFE method} \cite{li2003new} is to solve interface problems on unfitted meshes by designing special shape functions with similar behavior to the exact solution around the interface. On the interface elements, piecewise polynomials are enforced to (weakly) satisfy jump conditions on interface.  {Early work on IFE method focused on second-order elliptic problems \cite{li2003new}, employing piecewise linear polynomial spaces with provable optimal approximation properties for the corresponding IFE spaces \cite{li2004immersed}. The method was subsequently extended to bilinear polynomial spaces on rectangular elements \cite{he2008approximation} and to problems with non-homogeneous interface jump conditions \cite{he2011immersed}. In \cite{lin2015partially}, the optimal error estimates for IFE method for elliptic problems were established by introducing penalty terms on interface edges which were later extended to other classes of problems. Over the past decade, the IFE method has been developed for elliptic problems \cite{guo2021solving, lin2015partially}, elasticity problems \cite{guo2019approximation, 2012LinZhang}, wave equations \cite{adjerid2023immersed,adjerid2020error,adjerid2019immersed}, Stokes equations \cite{2021ChenZhang, jones2021class}, and Navier-Stokes equations \cite{2024ChenZhang, 2022WangZhangZhuang}, to name only a few.}
In recent years, the IFE method has been extended to higher order for solving elliptic problems through weak enforcement \cite{adjerid2018higher}, least-squares construction \cite{adjerid2017high},  Cauchy extension \cite{adjerid2020enriched,guo2019higher,2026ZhuangZhangSarkisLin}, and Frenet transformation \cite{2024AdjeridLinMeghaichi,2025AdjeridLinMeghaichi}. The main obstacle of high-order IFE construction is that {polynomials cannot, in general, satisfy interface conditions exactly on arbitrary curves}. The least-squares approximation provides an efficient approach to weakly impose the interface jump conditions \cite{adjerid2017high, 2021ChenZhang, 2024ChenZhang}.

 In this paper, we introduce a modified $C^0$ interior penalty method with  {$\mathcal{P}_p$} immersed  finite elements for solving the biharmonic interface problem \eqref{eq: bihar}.  High-order  {$\mathcal{P}_p$} immersed finite element spaces are constructed through a recently developed least-squares approach \cite{adjerid2017high,2021ChenZhang}. We analyze key properties of the new IFE space, such as unisolvency and partition of unity, and verify their approximation capability. These spaces are incorporated into a modified $C^0$ interior penalty method, in which additional penalty terms are imposed along the interface curves for stabilization.  We prove the well-posedness of the numerical scheme and present extensive numerical examples to demonstrate the optimal convergence in $L^2$, semi-$H^1$ and semi-$H^2$ norms. 

The rest of the paper is organized as follows: Section \ref{sec:note} introduces notations and assumptions used throughout the text.  Section \ref{sec:space} develops the high-order immersed finite element spaces and establishes their fundamental properties. Section \ref{sec:scheme} presents the modified C$^0$ interior penalty scheme and its theoretical analysis on well-posedness. Section \ref{sec:experiment} reports numerical experiments to assess the performance of the method. Section \ref{sec:conclu} concludes with a brief summary.  

\section{Notations and Assumptions}
\label{sec:note}
In this section, we introduce some notations and assumptions to be used in this paper. Let $\tilde{\Omega}\subset \Omega$, we denote the standard Sobolev space $W^{k,p}(\tilde \Omega)$ with norm $\|\cdot\|_{W^{k,p}(\tilde \Omega)}$ and semi-norm $|\cdot|_{W^{k,p}(\tilde \Omega)}$ for $k\geq 1$ and $1 \leq p \leq \infty$. For $p=2$, we write $H^k(\tilde \Omega)$ with norm $\|\cdot\|_k$ and semi-norm $|\cdot|_k$. If $\tilde\Omega \cap \Gamma\neq \emptyset$, we set $\tilde\Omega^{s} = \tilde\Omega\cap \Omega^{s}$ ($s=+,-$) and define the broken Sobolev space
\begin{equation}
	PW^{k,p}(\tilde\Omega)=\{u:u|_{\tilde\Omega^s}\in W^{k,p}(\tilde\Omega^s),s=+,-;~ u \text{ satisfies } \eqref{equ:jump1}-\eqref{equ:jump4}\}.	
\end{equation}
For $p=2$, we denote $PW^{k,2}(\tilde\Omega)$ by $PH^{k}(\tilde\Omega)$. For matrix value functions $\mathbf{u}$ with $\mathbf{u}_{i,j}\in W^{k,p}(\tilde \Omega)$, we denote the corresponding Sobolev space $\mathbf{W}^{k,p}(\tilde \Omega)$ and define
\begin{equation}
	\|\mathbf{u}\|_{\mathbf{W}^{k,p}(\tilde \Omega)} = \sum_{i,j} \| \mathbf{u}_{ij} \|_{W^{k,p}(\tilde \Omega)}.
\end{equation}
Let $\mathcal{T}_h$ be a shape-regular \cite{brenner2008mathematical} triangular mesh on $\Omega$ with mesh size 
$h=\max_{T\in \mathcal{T}_h} \{h_T\}$ where  $h_T$ is the diameter of the element $T \in \mathcal{T}_h$. Define the set of interface elements $\mathcal{T}_h^i = \{T \in \mathcal{T}_h: T\cap \Gamma \neq \emptyset\}$, and the non-interface elements 
$\mathcal{T}_h^n = \mathcal{T}_h \backslash \mathcal{T}_h^i$. For $T \in \mathcal{T}_h^i$, we let $T^{\pm}=T \cap \Omega^{\pm}$, and $\Gamma_T=\Gamma \cap T$. 

Let $\mathcal{E}_h$ denote the set of all edges of the mesh $\mathcal{T}_h$, and $\mathring{\mathcal{E}_h}$ the interior edges. Denote $\mathcal{E}_h^i$ and $\mathcal{E}_h^n = \mathcal{E}_h \backslash \mathcal{E}_h^i$ the set of interface edges and non-interface edges, respectively. For each $e\in\mathcal{E}_h^i$, define $e^{\pm}=e\cap \Omega^{\pm}$. Let $\mathcal{F}_h^i$ be the set of edges belonging to interface elements. For any $e\in\mathring{\mathcal{E}_h}$
denote the adjacent elements by $T_e^1$ and $T_e^2$, and define their jump and average by
\begin{equation}
\jump{w}_{e}=w|_{T_{e}^1}-w|_{T_{e}^2},~~~\aver{w}_{e}=\frac{1}{2}\left(w|_{T_{e}^1}+w|_{T_{e}^2}\right).
\end{equation}
The unit normal $\mathbf{n}_e$ is oriented from $T_e^1$ to $T_e^2$, and the tangent vector $\mathbf{t}_e$ is taken counterclockwise along $e$. These notations are illustrated in Figure \ref{fig:fic}. 

We assume that the interface $\Gamma$ is $C^2$ smooth and satisfies the following assumptions:
\begin{enumerate}
	\renewcommand{\labelenumi}{\textbf{(H\theenumi)}}
	\item {For every edge $e$, the intersection $\Gamma \cap e$ consists of at most one isolated point unless $e \subset \Gamma$.}
	\item {For any element $K$, the interface $\Gamma$ intersects the boundary $\partial K$ at most two distinct edges of $K$.}
\end{enumerate}
\begin{remark}
 Vertex intersections are allowed. In particular, if $\Gamma$ intersects $\partial K$ at two vertices of an element $K$, then these intersection points are regarded as lying on two distinct edges of $K$.
\end{remark}

%

The smoothness of $\Gamma$ ensures the following lemma holds:
\begin{lemma}[r-tubular neighborhood]
\label{lemma:rt}
Let $\Gamma$ be a regular, simple, $\mathcal{C}^{2}$ curve. For each point $X \in \Gamma$, let $N_{x}(r)$ to be the line segment of length $2 r$ centered at $x$ and perpendicular to $\Gamma$. Then, there exists $r>0$ such that for any two points $X, Y \in \Gamma, X \neq Y$, the line segments $N_{X}(r)$ and $N_{Y}(r)$ are disjoint. We define the $r$-tubular neighborhood of $\Gamma$ by $U_{\Gamma}(r)=\cup_{X \in \Gamma} N_{X}(r)$.	
\end{lemma}

As in most cases, our analysis in this paper is conducted on the mesh $\mathcal{T}_h$ with sufficiently small $h$. To make it precise, inspired by \cite{guzman2017finite}, we denote the radius of the tubular neighborhood of the interface $\Gamma$ by $r$ and then restrict $h<r/2$.

\begin{figure}[htbp]
 \centering
 \includegraphics[width=.8\textwidth]{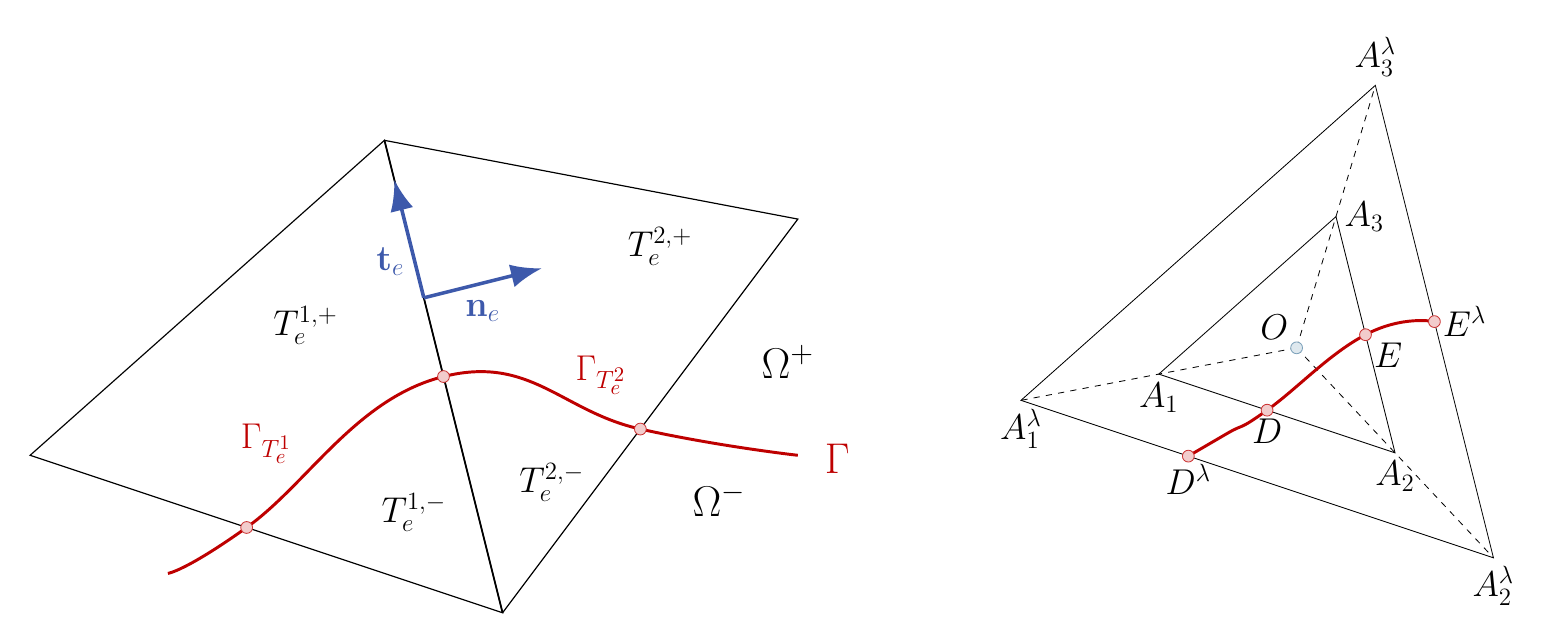}
 \caption{Notations of this article (Left) and fictitious element (Right)}
 \label{fig:fic}
\end{figure}

\section{High-Order Immersed FEM Space}
\label{sec:space}
In this section, we construct IFE spaces that approximately satisfy the biharmonic interface conditions \eqref{equ:jump1}-\eqref{equ:jump4}. 

For each triangular element $T \in \mathcal{T}_h$, let $\mathcal{P}_p(T)$ be the standard polynomial space of degree at most $p$ on $T$. Let  $\{A_i: i\in\mathcal{I}_p\}$ be the set of Lagrange nodes with set $\mathcal{N}_p(T) = \{A_i\}_{i\in\mathcal{I}_p}$ and $\mathcal{I}_p$ the index set. For example, when $p=2$, we have $|\mathcal{I}_p|=6$, with $A_1$, $A_2$, $A_3$ are vertices of $T$, and $A_4$, $A_5$, $A_6$ are midpoints of its edges.
\subsection{Least-Squares Biharmonic IFE Spaces}
On each non-interface element $T \in \mathcal{T}_h^n$, the local IFE space on $T$ is identical to the standard Lagrange finite element space, i.e. $\mathcal{P}_p(T)$. Let $\psi^p_{i,T} \in \mathcal{P}_p (T)$ be the Lagrange shape functions such that 
\[
\psi^p_{i,T}(A_j)=\delta_{ij},~~~\forall~  i,j\in\mathcal{I}_p,
\]
where $\delta_{ij}$ is Kronecker delta function. In this case, the local IFE space on $T$ is given by 
\[S_h^{p}(T) := \text{span}\{ \psi^p_{i,T}: i\in \mathcal{I}_p\}, ~ \text{for }T \in \mathcal{T}_h^n.\]

On each interface element $T\in\mathcal{T}_h^i$,  the interface conditions are enforced in a least-squares sense. We denote $\mathcal{I}_p^{\pm}=\{i \in \mathcal{I}_p: A_i\in T^{\pm}\}$ the index set of Lagrange nodes located in $T^\pm$. Moreover, we define the tensor polynomial space $\mathcal{S}_p(T)=\mathcal{P}_p(T) \times \mathcal{P}_p(T)$ and the piecewise polynomial space
\begin{equation}
	\mathcal{H}_p(T)=\{ v\in L^2(T): v|_{T^+} \in \mathcal{P}_p(T^+)\text{, }v|_{T^-} \in \mathcal{P}_p(T^-)\}.
\end{equation}

The spaces $\mathcal{S}_p(T)$ and $\mathcal{H}_p(T)$ are isomorphic, i.e.,  there exists a natural isomorphism $\mathcal{M}_T:\mathcal{S}_p(T)\to \mathcal{H}_p(T)$ given by
\begin{equation}
	\mathcal{M}_T(u,v)=\left\{\begin{array}{l}
	u, \text { on } T^{+} \\
	v, \text { on } T^{-}, 
	\end{array}\right.  \forall (u,v)\in \mathcal{S}_p(T).
\end{equation}

The local IFE space on the interface element is constructed as a subspace of $\mathcal{H}_p(T)$. For ease of exposition, we construct the IFE space as a subspace of $\mathcal{S}_p(T)$ and then map into $\mathcal{H}_p(T)$ through $\mathcal{M}_T$. It is straightforward to verify that the following functions form a basis of $\mathcal{S}_p(T)$:
\begin{equation}
	\xi^p_{i, T}=\left\{\begin{array}{l}
	\left(\psi^p_{i, T}, 0\right), \text { if } i \in \mathcal{I}_p^{+} \\
	\left(0, \psi^p_{i, T}\right), \text { if } i \in \mathcal{I}_p^{-},
	\end{array} \quad \eta^p_{i, T}=\left\{\begin{array}{l}
	\left(0, \psi^p_{i, T}\right), \text { if } i \in \mathcal{I}_p^{+} \\
	\left(\psi^p_{i, T}, 0\right), \text { if } i \in \mathcal{I}_p^{-}.
	\end{array}\right.\right.
\end{equation}
We may write $\mathcal{S}_p = \mathcal{S}_{p,1} \oplus \mathcal{S}_{p,2}$ where $\mathcal{S}_{p,1} = \text{span}\{\xi^p_{i,T},i\in\mathcal{I}_p\}$ and $\mathcal{S}_{p,2} = \text{span}\{\eta^p_{i,T},i\in\mathcal{I}_p\}$. To enforce interface conditions, we define bilinear form $\mathcal{J}^p_{\lambda}(\cdot,\cdot): \mathcal{S}_p(T)\times \mathcal{S}_p(T)\mapsto \mathbb{R}^{+} \cup \{0\}$, which measures the violation of the interface conditions:
%
\begin{equation}
	\label{equ:bilinearform2}
	\begin{split}
		\mathcal{J}^p_\lambda(u,v)=&\omega_0\int_{\Gamma_T^\lambda}\jump{u}\jump{v}ds+\omega_1h^2\int_{\Gamma_T^\lambda}\jump{\partial_n u}\jump{\partial_n v}ds+\omega_2h^4\int_{\Gamma_T^\lambda}\jump{\beta \partial_{nn} u}\jump{\beta  \partial_{nn} v}ds\\
		&+\omega_3h^6\int_{\Gamma_T^\lambda}\jump{\beta( \partial_n\Delta u+\partial_{ntt}u )}\jump{\beta ( \partial_n\Delta v+{\partial_{ntt}  v} )}ds.
	\end{split}	
\end{equation}
Here we adopted the idea of fictitious element introduced in \cite{zhuang2019high}, which constructs local IFE spaces on an enlarged element to alleviate ill-conditioning caused by extremely small-cut subelements. Specifically, for each $T\in \mathcal{T}_h^i$, we define the fictitious element 
\[T_{\lambda}=\{ X\in \mathbb{R}^2:\exists Y\in T\text{ s.t. }\overrightarrow{O X}=\lambda \overrightarrow{O Y} \},\] 
where $O$ is the barycenter of $T$. The extended interface segment is then $\Gamma_T^{\lambda}=T_{\lambda}\cap \Gamma$. 
{The scaling with respect to $h$ in \eqref{equ:bilinearform2} reflects the order of derivatives in each interface condition and ensures proper balance in the least-squares functional.}
An illustration of the fictitious element is shown on the right hand side of Figure \ref{fig:fic}.

The bilinear functional $\mathcal{J}^p_{\lambda}(\cdot,\cdot)$ defines a semi-norm on $\mathcal{S}_p$, denoted by $|\cdot|_{\mathcal{J}^p_{\lambda}}$. In fact, for $p=2$ or $3$, it induces a norm, as will be shown later. 

Given nodal values $\mathbf{v}=(v_i)_{i\in \mathcal{I}_p}$, the associated  $\mathcal{S}_p$ IFE function on $T$ takes the form
\begin{equation}\label{eq: IFE nodal}
	\varphi_{T}^p|_{\mathbf{v}} =\sum_{i\in \mathcal{I}_p} v_i \xi^p_{i,T}+\sum_{i\in \mathcal{I}_p}	c_i \eta^p_{i,T},
\end{equation}
where  {we use ``$|_{\mathbf{v}}$'' to denote the shape function given nodal value $\mathbf{v}$ and} the unknown coefficient vector $\mathbf{c}=(c_i)_{i\in \mathcal{I}_p}$ is determined by minimizing the violation of interface conditions $|\cdot|_{\mathcal{J}_\lambda^p}$. The precise statement is in Theorem \ref{thm:lsq}. We test \eqref{eq: IFE nodal} against $\eta^p_{i,T}$ to enforce $\mathcal{J}^p_\lambda(\varphi_{T}^p|_{\mathbf{v}},\eta^p_{i,T})=0$, $i=1,2,...,|\mathcal{I}_p|$ leading to the following linear system:
\begin{equation}
	\mathbf{A}^{p,\lambda} \mathbf{c}=-\mathbf{B}^{p,\lambda} \mathbf{v},
\end{equation}
with
\begin{equation}
	\mathbf{A}^{p,\lambda}_{i j}=\left(\mathcal{J}^p_{\lambda}\left(\eta_{j, T}, \eta_{i, T}\right)\right)_{i, j \in \mathcal{I}_p} \in \mathbb{R}^{|\mathcal{I}_p| \times|\mathcal{I}_p|}, \quad \mathbf{B}^{p,\lambda}_{i j}=\left(\mathcal{J}^p_{\lambda}\left(\xi_{j, T}, \eta_{i, T}\right)\right)_{i, j \in \mathcal{I}_p} \in \mathbb{R}^{|\mathcal{I}_p| \times|\mathcal{I}_p|}.		
\end{equation}
The IFE functions of space $\mathcal{H}_p$ are then constructed by mapping $\varphi^p_{T}|_{\mathbf{v}}$ back via $\mathcal{M}_T$. In summary, the IFE basis functions on $T \in \mathcal{T}_h^i$ is defined by:
\begin{equation}
	\phi^p_{i,T} := \mathcal{M}_T\varphi^p_{T}|_{\mathbf{e}_i},~~~~i\in \mathcal{I}_p,
\end{equation}
where $\mathbf{e}_i$ is the canonical basis vector in $\mathbb{R}^{|\mathcal{I}_p|}$.  It is straightforward to verify $\phi^p_{i,T}(A_j)=\delta_{ij}$ for $i,j \in \mathcal{I}_p$. 
Examples of $\mathcal{P}_2$ and $\mathcal{P}_3$ IFE basis functions are shown in Figure \ref{fig:basis}. In the first row, the third, fifth and sixth $\mathcal{P}_2$ IFE basis functions are presented, while the second row lists the third, fifth and tenth $\mathcal{P}_3$ basis functions. It follows that the pointwise continuity of basis functions along the interface $\Gamma_T$ cannot be guaranteed. In fact, enforcing two distinct polynomials to coincide pointwise on nontrivial curve is generally impossible. Instead, the IFE basis functions are weakly continuous across the interface.

\begin{figure*}[!htbp]
\centering
\subfigure{
    \begin{minipage}[t]{0.33\linewidth}
        \centering
        \includegraphics[width=\textwidth]{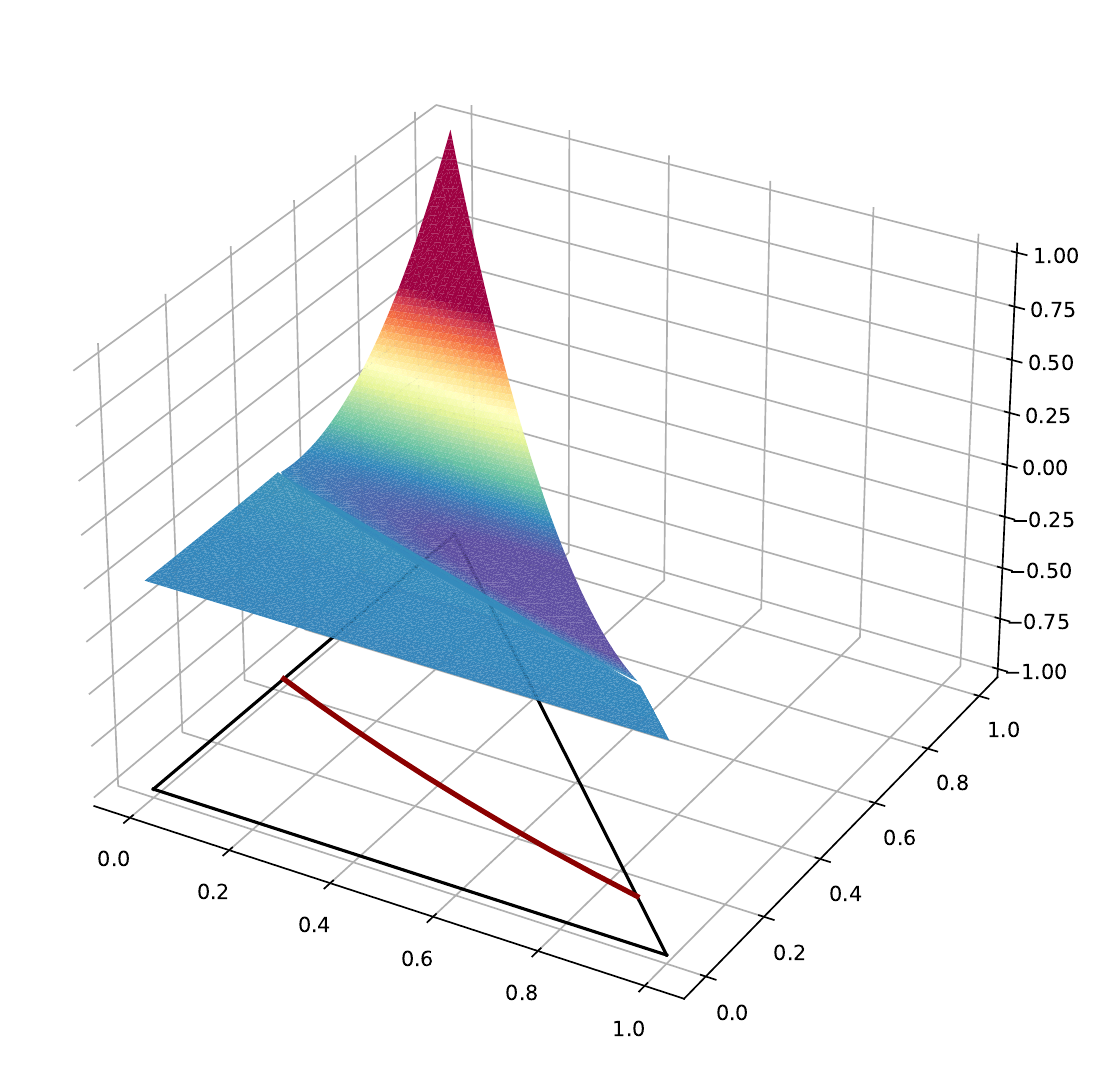}\\
        \vspace{0.02cm}
        \includegraphics[width=\textwidth]{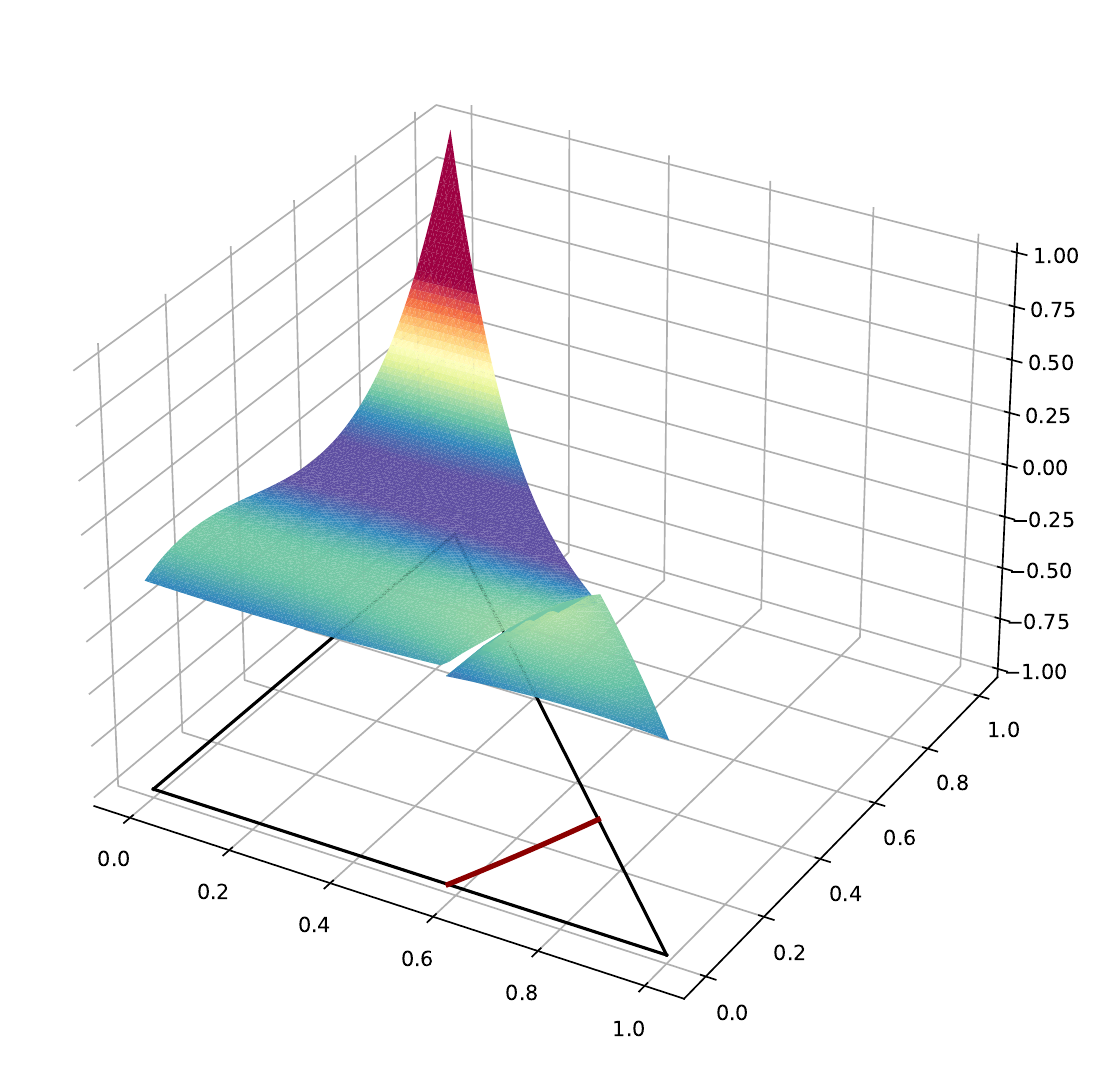}\\
        \vspace{0.02cm}
    \end{minipage}%
}%
\subfigure{
    \begin{minipage}[t]{0.33\linewidth}
        \centering
        \includegraphics[width=\textwidth]{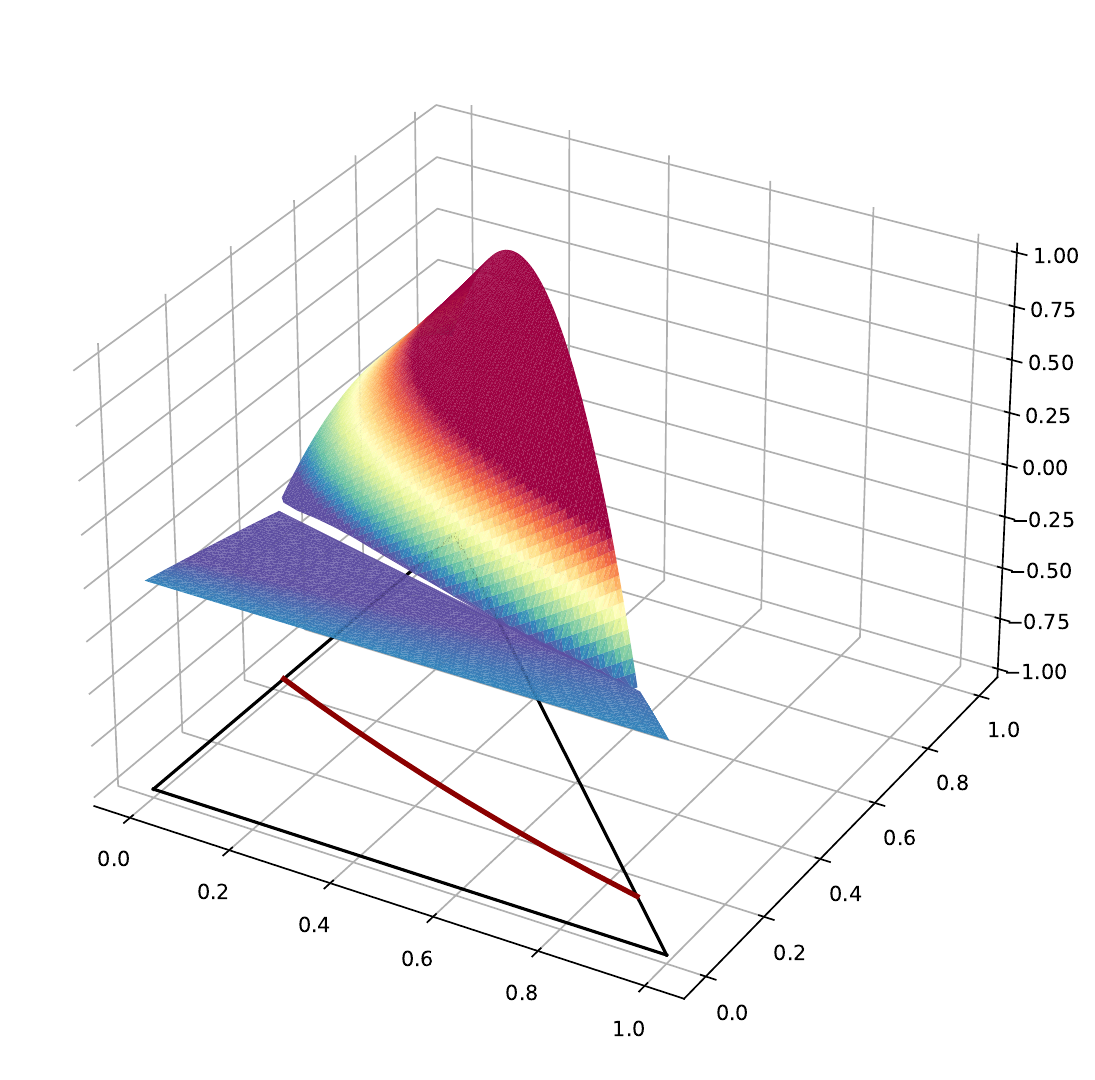}\\
        \vspace{0.02cm}
        \includegraphics[width=\textwidth]{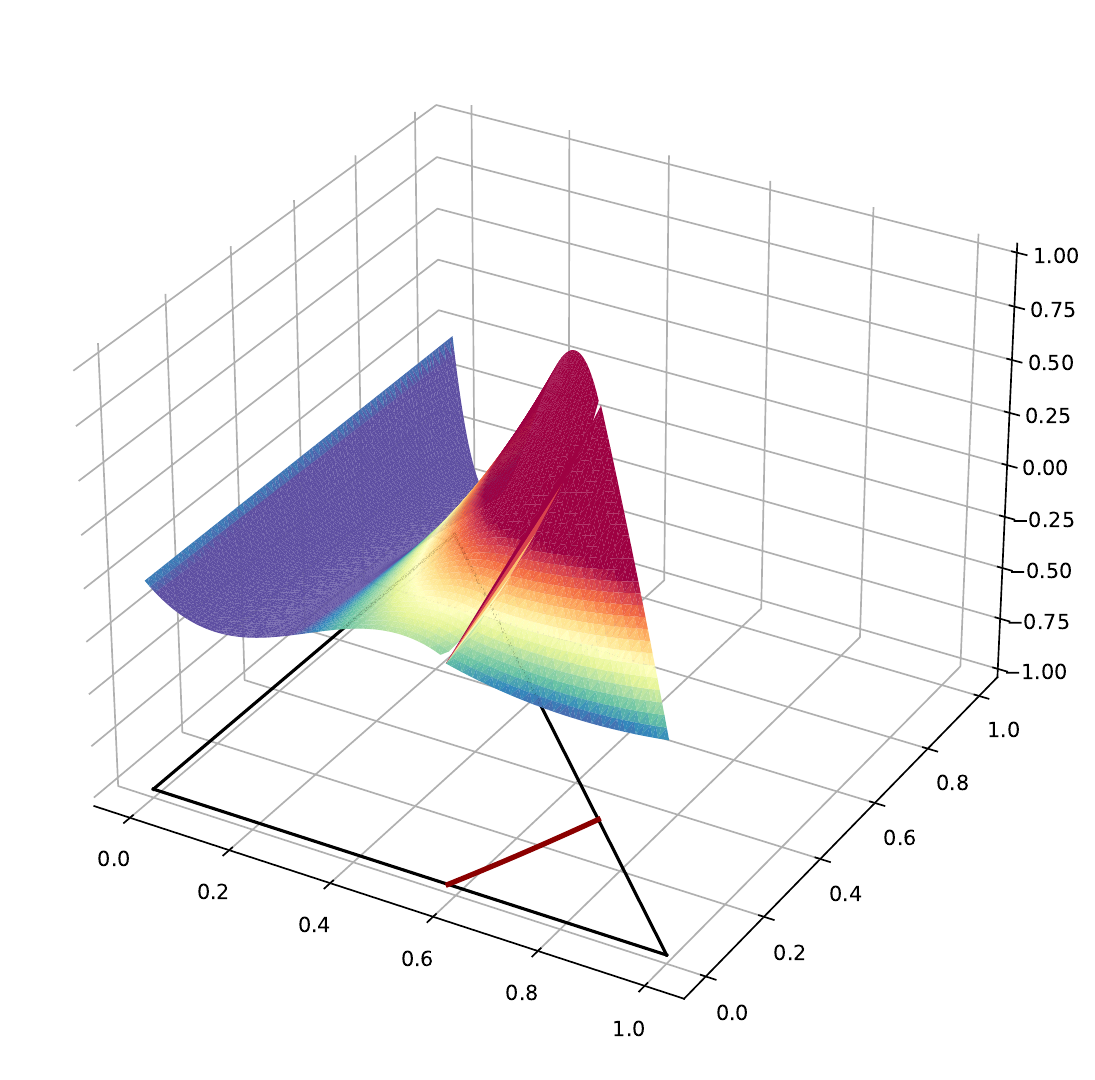}\\
        \vspace{0.02cm}
    \end{minipage}%
}%
\subfigure{
    \begin{minipage}[t]{0.33\linewidth}
        \centering
        \includegraphics[width=\textwidth]{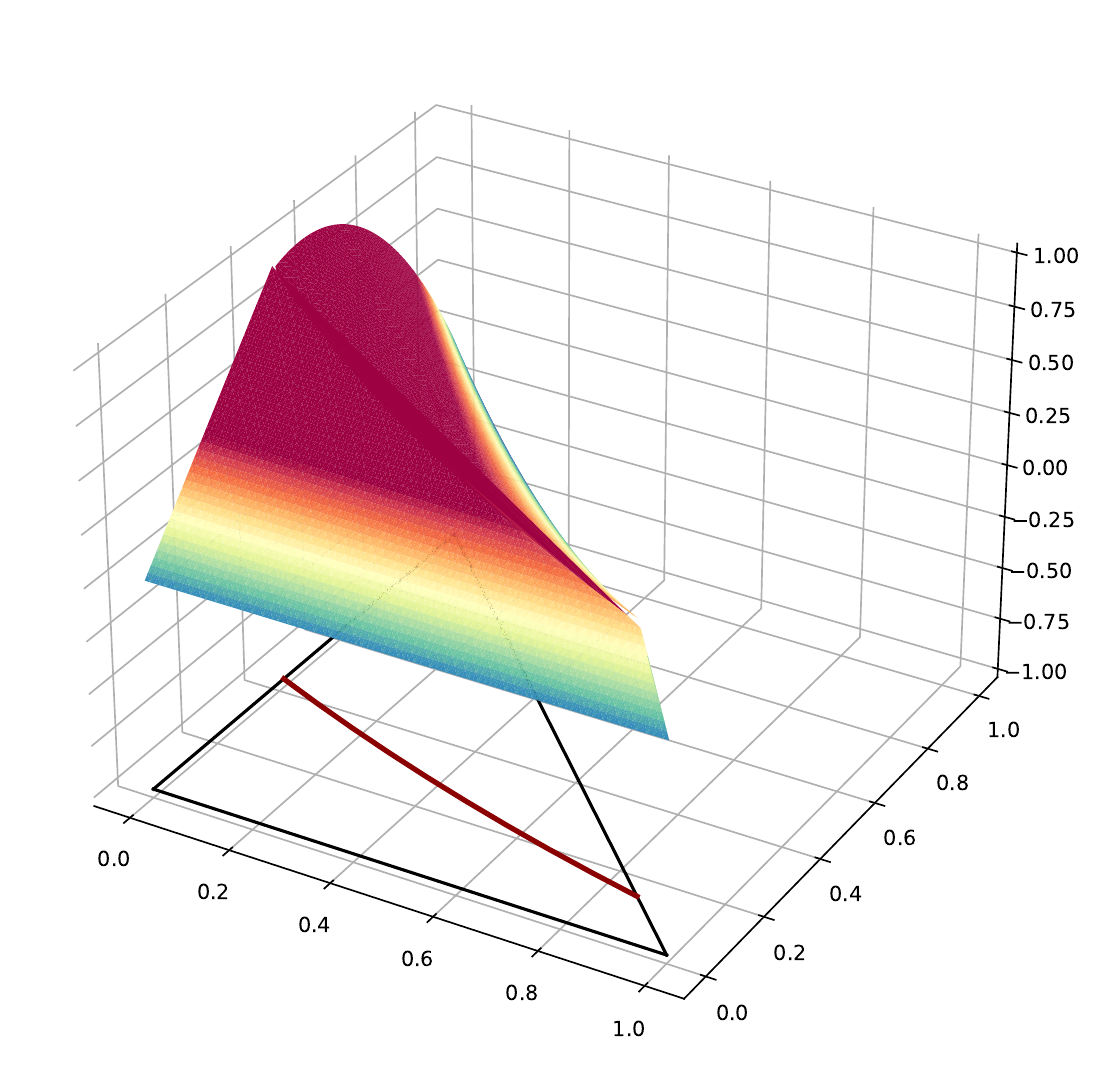}\\
        \vspace{0.02cm}
        \includegraphics[width=\textwidth]{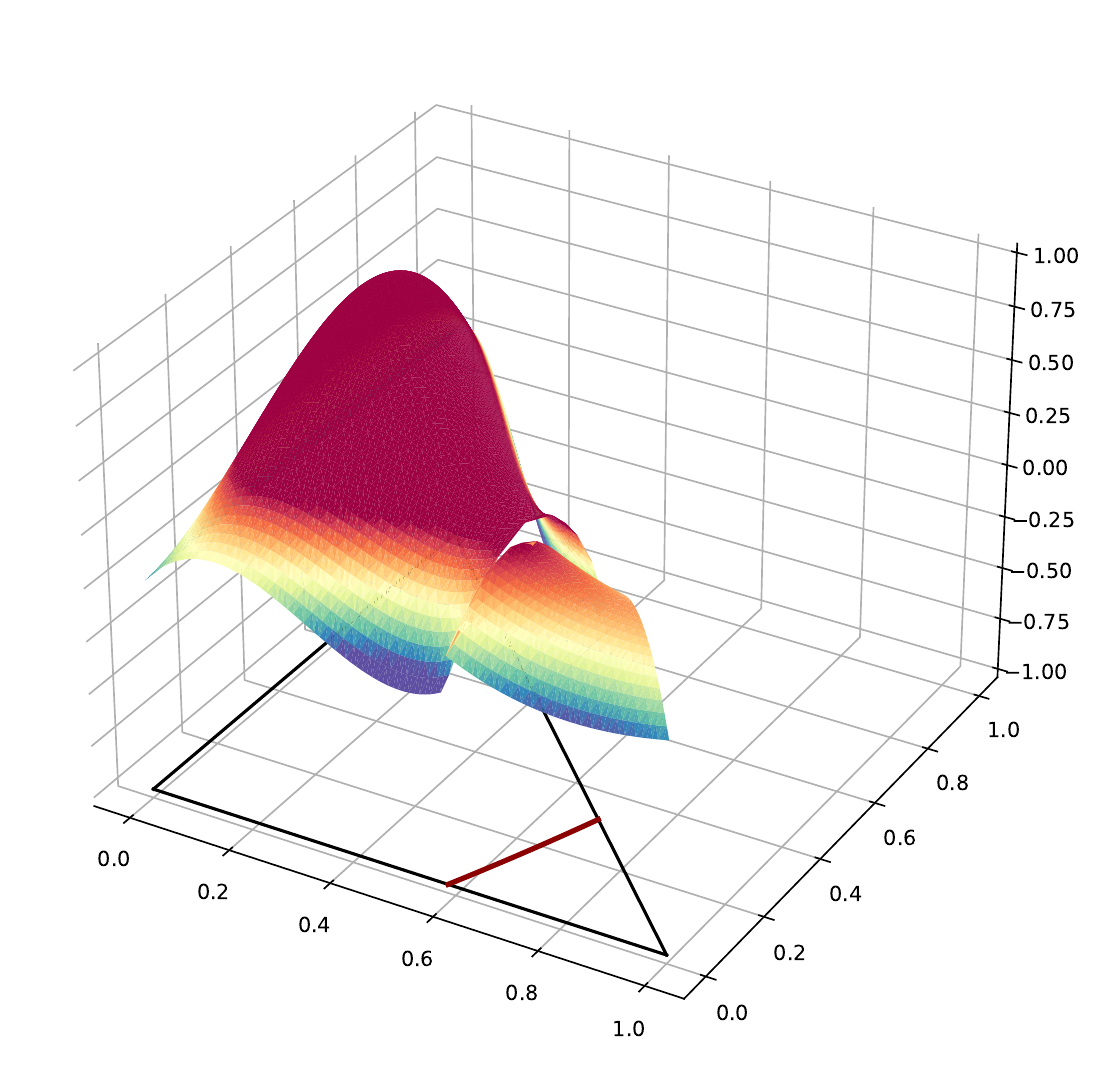}\\
        \vspace{0.02cm}
    \end{minipage}%
}%
\centering
\vspace{-1mm}
\caption{First Row: 3rd, 5th, 6th $\mathcal{P}_2$ basis functions; Second Row: 3rd, 5th, 10th $\mathcal{P}_3$ basis functions.}
\label{fig:basis}
\end{figure*}

For each $T\in\mathcal{T}_h$, we define the local finite element space $S^p_h(T)$ by:
\begin{equation}\label{eq: global space}
	S^p_h(T):=\left\{\begin{array}{l}
	\text{span}\{\psi_{i,T}^p, ~i\in \mathcal{I}_p\}, ~\text { if } T \in \mathcal{T}_h^n \\
	\text{span}\{\phi_{i,T}^p, ~i\in \mathcal{I}_p\}, ~\text{ if } T \in \mathcal{T}_h^i.
	\end{array}\right.
\end{equation}
The global IFE space $S^p_h(\Omega)$ is then defined by
\begin{equation}
	S^p_h(\Omega)=\left\{u \in L^2(\Omega): ~u~\text{satisfies conditions \textbf{(C1)}-\textbf{(C4)}}\right\}.
\end{equation}
where
\begin{description}
\item[(C1)]~ $u|_T\in S^p_h(T)$, for all $T\in\mathcal{T}_h$.
\item[(C2)]~ $u$ is continuous on every non-interface edge $e\in \mathcal{E}_h^n$.
\item[(C3)]~ $u$ is continuous at all nodals $A_i$, $i\in\mathcal{I}_p$ for all $T\in\mathcal{T}_h$.
\item[(C4)]~ $u|_{\partial \Omega}=0$.
\end{description}

\begin{figure}[htbp]
 \centering
 \includegraphics[width=.9\textwidth]{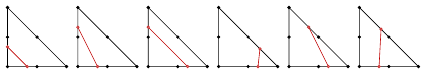}
 \caption{The exemplified cut cases when interface is a line. The sub-figures 1-3 are Type I elements for $\mathcal{P}_2$, the sub-figures 4-6 are Type II elements for $\mathcal{P}_2$.}
 \label{fig:line_uni}
\end{figure}


\begin{remark}
This least-squares functional \eqref{equ:bilinearform2} is sufficient for polynomial degree $p\le 3$. For higher degree such that $p \geq 4$, an appropriate extended jump conditions across interface needs to be employed in the least-squares functional. 
\end{remark}

\subsection{Property of Biharmonic IFE Spaces}
In this subsection, we present some basic properties for the biharmonic IFE space \eqref{eq: global space}. 
\begin{theorem}[Unisolvence]
	\label{thm:unisol}
	On every interface element $T\in \mathcal{T}_h^i$, for polynomial degrees $p=2$ and $p=3$, given any nodal values $\mathbf{v}$, there exists a unique IFE shape function in the form of \eqref{eq: IFE nodal}.
\end{theorem}
\begin{proof}
	The proof of existence follows arguments similar to Theorem  {3.1} of \cite{2021ChenZhang}. For uniqueness, let $\mathcal{K}^p_{\lambda}$ denote  the null space of  {$\mathcal{J}_\lambda^p$, which is defined by $\mathcal{K}^p_{\lambda}:=\{ \phi\in \mathcal{S}_p(T): |\phi|_{\mathcal{J}_\lambda^p}=0 \}$. We note that all functions with vanishing nodal values live in space $\mathcal{S}_{p,2}$, then it suffices to prove $ \mathcal{S}_{p,2} \cap \mathcal{K}^p_{\lambda}=\{(0,0)\}$ for uniqueness.}
	
	We prove in different categories of curve types for $\Gamma_T^\lambda$.
	\begin{case}
		$\Gamma_T^\lambda$ is a straight line.
	\end{case}
	 For this degenerate case, let $u=(u_1,u_2) \in \mathcal{S}_{p,2} \cap \mathcal{K}^p_{\lambda}$. 
	 {Since $u\in \mathcal{S}_{p,2}$}, we may write $u=\sum_{i=1}^{|\mathcal{I}_{ {p}}|}c_i\eta_{i,T}$. 
	Because $\Gamma_T^\lambda$ is a line, the condition $u \in \mathcal{K}^p_{\lambda}$ implies
	\begin{equation}
	\jump{u}\equiv 0,~~ \jump{\partial_n u}\equiv 0, ~~\jump{\beta \partial_{nn} u }\equiv 0~~ \text{and}~~ \jump{\beta\left( \partial_n \Delta u+\partial_{ntt} u\right)}\equiv 0   {\text{~~ on }\Gamma_T^\lambda}.
	\end{equation}
	 {Each of the above condition is equivalent to multiple number of linear equations of $\mathbf{c}$ since $\jump{u}$, $\jump{\partial_n u}$, $\jump{\beta \partial_{nn} u }$, $\jump{\beta\left( \partial_n \Delta u+\partial_{ntt} u\right)}$ are $p$, $p-1$, $p-2$, $p-3$ degree polynomials on the line segment $\Gamma_T^\lambda$, respectively. To enforce those conditions, we need $(p+1)+p+(p-1)+(p-2)=4p-2$ equations in total. Thus the coefficient $\mathbf{c}$ should satisfy a linear system $\mathbf{Mc}=\mathbf{0} $ where $\mathbf{M} \in \mathbb{R}^{|\mathcal{I}_p|\times|\mathcal{I}_p|}$ since $|\mathcal{I}_p|$ is identical to $4p-2$ for (only) $p=2,3$.} The structure of $\mathbf{M}$ depends on how the interface segment intersects the element $T$.  {A similar idea is used in the $\mathcal{P}_1$ construction of IFE spaces, see e.g. \cite{li2003new}. Following the convention of IFE literature,  $T$ is called a type I element if} the interface passes through two legs of the right angle; otherwise, it is called a type II element. Some exemplified cutting configurations are illustrated in Figure \ref{fig:line_uni}  {for $p=2$}. In each case, uniqueness follows by verifying that $\det(\mathbf{M})\neq 0$ through direct computations. For instance, in the  {third} configuration of Figure \ref{fig:line_uni},  {we work on reference elements and denote the two intersection points of the interface and element by $D:=(d,0)$ and $E:=(0,e)$. The matrix $\mathbf{M}$ can be constructed by requiring}
 {
\begin{equation*}
	\jump{u}(D)=0,~\jump{u}(E)=0,~\jump{u}\left(\frac{D+E}{2}\right)=0,~\jump{\partial_n u}(D)=0,~\jump{\partial_n u}(E)=0,~\jump{\beta \partial_{nn} u }\left(\frac{D+E}{2}\right)=0.
\end{equation*}
} 
 {The} determinant can be explicitly written as:
	\begin{align}
		\det(\mathbf{M})=32 (d^2 + e^2)^2 (P_1+\rho P_2)> 0,	
	\end{align}
	 {for} $1/2\leq d \leq 1$ and $1/2\leq e \leq 1$, $\rho=\beta^+/\beta^-$, and $P_1$ are $P_2$ are polynomial of $d$ and $e$ which cannot be equal to $0$ at the same time:
	\begin{equation}
		\begin{split}
			P_1=&-d^3 (1 - 2 e)^2 e - 2 d^2 (-2 + e) e^3 - e^4 + d e^3 (-1 + 4 e) + d^4 (-1 + 4 e - 2 e^2) \geq P_1(1/2,1/2) \geq 0,\\
			P_2=&2 d^4 (-1 + e)^2 + d^3 (1 - 2 e)^2 e + 2 d^2 (-1 + e)^2 e^2 + 2 e^4 +  d (e^3 - 4 e^4) \geq P_2(1,1) \geq 0.
		\end{split}
	\end{equation}
	The remaining case can be verified using the same strategy, although the computation are tedious to be included here. Hence, we present the remaining cases in the  {Appendix \ref{Case1}}. This proves the positive definiteness of $\mathcal{J}^p_{\lambda}$ on $\mathcal{S}_{p,2}$, which means $\mathcal{J}^p_{\lambda}(\cdot,\cdot)$ is an inner product on $\mathcal{S}_{p,2}$. Consequently, the matrix  {$\mathbf{A}^{p,\lambda}$} is positive definite and therefore invertible.

	\begin{case}
		 $\Gamma_T^\lambda$ is either a non-algebraic curve or an algebraic curve with degree greater than $p$.
	\end{case}
	In this case, let $(u_1,u_2) \in \mathcal{S}_{p,2} \cap \mathcal{K}^p_{\lambda}$, then we claim that $(u_1,u_2) \equiv(0,0)$. Indeed, note that $u_1-u_2$ is a polynomial of degree at most $p$. It is impossible for such a polynomial to vanish almost everywhere on a non-algebraic curve, or an algebraic curve with degree greater than $p$. Therefore, we must have $u_1-u_2 \equiv 0$ on $T_{\lambda}$, which implies $u_1=u_2$.  Since  {$u_1,\,u_2\in\mathcal{P}_p(T_{\lambda})$ and} both vanish at all nodes $A_i$, $i\in \mathcal{I}_p$, it follows that $u_1=u_2\equiv 0$. 	
	 Hence, $\mathcal{J}^p_{\lambda}(\cdot,\cdot)$ is an inner product on $\mathcal{S}_{p,2}$. Again, we conclude that $\mathbf{A}$ is positive definite and thus invertible.
	\begin{case}
		$\Gamma_T^\lambda$ is an algebraic curve of degree larger than $p/2$ but less than or equal to $p$. 
	\end{case}
	In this case, let $(u_1,u_2) \in \mathcal{S}_{p,2} \cap \mathcal{K}_{\lambda}$. We may write 
	\[u_1-u_2=L^kQ\]
	 where $L(x,y)$ is a polynomial whose zeros are given by  $\Gamma_T^\lambda$ and  $k \geq 1$. Without loss of generality,  assume that $L$ is irreducible and relatively prime to the non-constant polynomial $Q$. Since  $u_1-u_2$ is of order at most $p$ and $\Gamma_T^\lambda$ is of degree larger than $p/2$, we can only have  $k<2$, i.e. $k=1$. Since $(u_1,u_2) \in  \mathcal{K}^p_{\lambda}$, we must have 
	 \[  {\partial_{n} (L^kQ) = 0,~~~\text{on }\Gamma_T^\lambda.} \]
	 By direct computation, we have 
	 \[\partial_{n} L^kQ = kL^{k-1} \partial_{n} (L)Q+L^k \partial_{n} Q.\] 
	 When $k=1$, this reduces to $\partial_{n} LQ = (\partial_{n} L)Q$ on $\Gamma_T^\lambda$. Here, $(\partial_n Q)L=0$ is  due to $L=0$ on $\Gamma_T^\lambda$. Since $L$ and $Q$ are relatively prime, it follows that $\partial_{n} L \equiv 0$ on $\Gamma_T^\lambda$.  {Then the implicit function theorem implies }$\partial_x L=\partial_y L=0$ on $\Gamma_T^\lambda$, contradicting the assumption that $\Gamma_T^\lambda$ is a nontrivial algebraic curve. Thus, the only possibility is  $k=0$ and $Q \equiv 0$, so that $u_1-u_2 {\equiv} 0$ on $T_{\lambda}$. The remaining conclusion follows as in the previous cases.
\end{proof}

\begin{remark}
	For general polynomial order $p$, the construction based on the bilinear form  \eqref{equ:bilinearform2} guarantees the existence of IFE shape functions. However, uniqueness can only be established when the interface curve is either non-algebraic or an algebraic curve of degree strictly greater than $p/2$.  {This can be also seen from our argument of Case 1 linear case, the number of constraints on coefficients $\mathbf{c}$, $4p-2$ will be less than dimension of $\mathbf{c}$ when $p>3$, which will result in a underdetermined system.} So for higher-order cases, it is natural to extend the interface jump conditions by incorporating additional terms in \eqref{equ:bilinearform2}.
\end{remark}

\begin{theorem}[Least Square Construction]
\label{thm:lsq}
	On every interface element $T\in \mathcal{T}_h^i$, given nodal values $\mathbf{v}$, the IFE shape function constructed by the above procedure minimizes the semi norm $|\cdot|_{\mathcal{J}^p_\lambda}$ over $\mathcal{S}_p$.
\end{theorem}
\begin{proof}
	The proof follows the same argument as Theorem 2.2 in \cite{2021ChenZhang}. The property is a consequence of the least-squares construction of the IFE shape functions. 
	\end{proof}


\begin{theorem}[Partition of Unity]
\label{thm:unique}
	On every interface element $T\in \mathcal{T}_h^i$, IFE basis functions satisfy the following partition of unity property:
	\begin{equation}
		\sum_{i\in \mathcal{I}_p} \phi^p_{i, T} \equiv 1,~~~~~p=2,3.
	\end{equation} 
\end{theorem}
\begin{proof}
	Let $\phi^p_T=\sum_{i\in \mathcal{I}_p}\phi^p_{i,T} {\in\mathcal{H}_p}$, so $\mathcal{M}_T^{-1}\phi^p_{T}\in \mathcal{S}_p(T)$. By  {$\phi^p_{i,T}(A_j)=\delta_{ij}$}, we have 
	\[\phi^p_{T}(A_i)=1,~i\in \mathcal{I}_p.\] 
	 {Thus $\phi^p_T$ is an IFE function with nodal value $\mathbf{v}=\mathbf{1}$. Then $\phi^p_T$ should admit the following decomposition} 
	\begin{equation}
		\psi^p_{T}=\mathcal{M}_T^{-1}\phi^p_T,~~ \psi^p_{T}=\sum_{i\in \mathcal{I}_p} \xi^p_{i,T}+\sum_{i\in \mathcal{I}_p} c_i\eta^p_{i,T},
	\end{equation}
	 {where $\mathbf{c}$ is a coefficient vector determined by the least-squares construction. On the other hand, $\mathbf{c}=\mathbf{1}$ is a valid solution. This is because we have $\sum_{i\in \mathcal{I}_p} \xi^p_{i,T}+\sum_{i\in \mathcal{I}_p} c_i\eta^p_{i,T} \equiv (1,1)$ on $T$ (due to the partition of unity results of standard Lagrange  basis functions) and subsequently $\mathcal{J}^p_\lambda(\psi^p_{T},\eta^p_{i,T})=0$ for $i=1,2,...,|\mathcal{I}_p|$ when $\mathbf{c}=\mathbf{1}$. }By uniqueness result in Theorem \ref{thm:unisol}, this solution is unique, so $\mathbf{c}=\mathbf{1}$. Therefore,
	\begin{equation}
		\phi^p_T=\mathcal{M}_T\left(\sum_{i\in \mathcal{I}_p} \xi^p_{i,T}+\sum_{i\in \mathcal{I}_p} \eta^p_{i,T}\right)\equiv 1, \text{ on }T.
	\end{equation}
\end{proof}

\section{Immersed $C^0$ Interior Penalty  Method}
\label{sec:scheme}
In this section, we derive the immersed $C^0$ interior penalty scheme for the biharmonic interface problem. 

\subsection{Derivation of Immersed $C^0$ Interior Penalty Method}
Let $u\in PH^4(\Omega)$ be the true solution of the problem \eqref{equ:interface}. Define 
\begin{equation}
	\begin{split}
		V_h=\{ v\in H^1_0(\Omega):&~v|_T \in H^2(T) ~\text{ if }~ T \in \mathcal{T}_h^n,~v|_{T^{\pm}} \in H^2(T^{\pm})~\text{ if }~ T \in \mathcal{T}_h^i\}.
	\end{split}	
\end{equation}
On each interface element $T\in\mathcal{T}_h^i$, multiplying the equation  \eqref{equ:bihm} by $v\in V_h$ and integrating by parts on the subelement $T^s$ ($s=\pm$) gives
\begin{equation}
\begin{aligned}
\int_{T^s} \Delta \left( \beta \Delta u\right)v ~dX 
&=-\int_{T^s} \nabla v \cdot\left(\nabla \beta \Delta u\right) ~dX +\int_{\partial T^s} v\left(\nabla \beta \Delta u \cdot \mathbf{n}\right) ~ds\\
&=\int_{T^s}\beta\nabla^2 u : \nabla^2 v ~dX + \int_{\partial T^s} \beta \left( {\partial_n \Delta u}\right) v ~ds - \int_{\partial T^s}\beta \left(\nabla^2 u\right) : \left(\nabla v \otimes n\right) ~ds,
\end{aligned}
\end{equation}
where $(a \otimes b)_{ij}=a_i b_j$. After some simplification, we obtain
\begin{equation}
\label{equ:localgreen}
\begin{aligned}
\int_{T^s}\left(\beta\Delta^{2} u\right) v ~dX = \int_{T^s} \beta\nabla^{2} u: \nabla^{2} v ~dX + \int_{\partial T^s}\left[\beta\partial_n \Delta u v-\beta  \partial_{nn} u \partial_n v-\beta \partial_{nt}u
\partial_t v\right] d s.
\end{aligned}
\end{equation}
Using the tangential derivative identity:
\begin{equation}
	\begin{aligned}
		\int_{\partial T^s} \partial_{nt} u\partial_{t} v ~ds = \int_{\partial T^s} \partial_t\left(\partial_{nt} uv\right) ~ds - \int_{\partial T^s} \partial_{ntt} u v ~ds,
	\end{aligned}
\end{equation}
together with the gradient theorem
\begin{equation}
	\begin{aligned}
		\int_{\partial T^s} \partial_t \left(\partial_{nt} uv\right) ~ds = 0.
	\end{aligned}
\end{equation}
The equation \eqref{equ:localgreen} becomes
\begin{equation}
\label{equ:localgreen2}
\begin{aligned}
\int_{T^s}\left(\beta\Delta^{2} u\right) v ~dX = \int_{T^s} \beta\nabla^{2} u: \nabla^{2} v ~dX - \int_{\partial T^s}\beta\partial_{nn}u \partial_n v d s + \int_{\partial T^s}\beta\left(\partial_{n}\Delta u +\partial_{ntt} u  \right) v d s.
\end{aligned}
\end{equation}
Similarly, on non-interface elements $T\in \mathcal{T}_h^n$, it holds \cite{2005BrennerSung}
\begin{equation}
\label{equ:localgreen3}
\begin{aligned}
\int_{T}\left(\beta\Delta^{2} u\right) v ~dX = \int_{T} \beta\nabla^{2} u: \nabla^{2} v ~dX - \int_{\partial T}\beta\partial_{nn}u \partial_n v d s + \int_{\partial T}\beta\left(\partial_{n}\Delta u +\partial_{ntt} u  \right) v d s.
\end{aligned}
\end{equation}
Summing \eqref{equ:localgreen2} over $s=+,-$ and all interface elements $T \in \mathcal{T}^i_h$,  together with summing \eqref{equ:localgreen3} over all $T \in \mathcal{T}^n_h$, yields
\begin{equation}
\begin{aligned}
\sum_{T\in \mathcal{T}_h}\int_{T} fv ~dX = 
&\sum_{T\in \mathcal{T}_h} \int_{T} \beta\nabla^{2} u: \nabla^{2} v ~dX \\
&- \sum_{e\in \mathring{\mathcal{E}}_h}\int_{e} \aver{\beta\partial_{nn}u }\jump{\partial_n v} d s + \sum_{e\in \mathring{\mathcal{E}}^i_h}\int_{e} \aver{\beta\left(\partial_{n}\Delta u +\partial_{ntt} u  \right)} \jump{v} d s\\
&- \sum_{T \in \mathcal{T}_h^i}\int_{\Gamma_T} \aver{\beta\partial_{nn}u }\jump{\partial_n v} d s + \sum_{T \in \mathcal{T}_h^i}\int_{\Gamma_T} \aver{\beta\left(\partial_{n}\Delta u +\partial_{ntt} u  \right)} \jump{v} d s.
\end{aligned}
\end{equation}
Finally, adding symmetric terms and stabilization terms following the convention of interior penalty methods, the immersed $C^0$ interior penalty Galerkin scheme for the biharmonic interface problem is defined as: find $u_h\in S_h^p(\Omega)$, such that:
\begin{equation}
	\label{equ:scheme}
	a_h(u_h,v_h)=L_{f}(v_h), ~~\forall~v_h \in S_h^p(\Omega),
\end{equation}
where $a_h(u_h,v_h) = A_h(u_h,v_h)+J_{h,u}(u_h,v_h)+J_{h,n}(u_h,v_h)$. The components of $a_h(u_h,v_h)$ are given by
\begin{equation}
	\begin{aligned}
		A_h(u,v) = & \sum_{T\in \mathcal{T}_h} \int_{T} \beta\nabla^{2} u: \nabla^{2} v ~dX,\\
		J_{h,u}(u,v) =&  - \sum_{e\in \mathring{\mathcal{E}}_h}\int_{e} \aver{\beta\partial_{nn} u }\jump{\partial_n v} d s- \sum_{e\in \mathring{\mathcal{E}}_h}\int_{e} \aver{\beta\partial_{nn} v }\jump{\partial_n u} d s+ \sum_{e\in \mathring{\mathcal{E}}^n_h}\frac{\sigma_u \beta}{|e|}\int_{e} \jump{\partial_n u}\jump{\partial_n v} d s\\
			&+ \sum_{e\in \mathring{\mathcal{E}}^i_h}\left(\frac{\sigma_u \beta^+}{|e^+|}\int_{e^+} \jump{\partial_n u}\jump{\partial_n v} d s+\frac{\sigma_u \beta^-}{|e^-|}\int_{e^-} \jump{\partial_n u}\jump{\partial_n v} d s\right) 
			+ \sum_{T \in \mathcal{T}_h^i}\frac{\sigma_u \aver{\beta}}{h_T}\int_{\Gamma_T} \jump{\partial_n u}\jump{\partial_n v} d s,
	\end{aligned}
\end{equation}
\begin{equation}
	\begin{aligned}
		 J_{h,n}(u,v) = & \sum_{e\in \mathring{\mathcal{E}}^i_h}\sigma_F\left( |e^+|\beta^+\int_{e^+} \jump{\partial_{nn} u }\jump{\partial_{nn} v } d s + |e^-|\beta^-\int_{e^-} \jump{\partial_{nn} u }\jump{\partial_{nn} v } d s \right)\\
		    &+\sum_{e\in \mathring{\mathcal{E}}_h^n \cap \mathcal{F}_h^i}\sigma_F|e|\beta\int_{e} \jump{\partial_{nn} u }\jump{\partial_{nn} v } d s+\sum_{e\in \mathring{\mathcal{E}}^i_h}\frac{\sigma_n \aver{\beta}}{|e|^3 } \int_{e} \jump{u}\jump{v} d s + \sum_{T \in \mathcal{T}_h^i}\frac{\sigma_n \aver{\beta}}{h_T^3 } \int_{\Gamma_T} \jump{u}\jump{v} d s.
	\end{aligned}
\end{equation}
with $\aver{\beta}= (\beta^+ + \beta^-)/2$.
In this scheme, stabilization terms are included both on interface edges and interface segments, since the IFE space does not enforce the interface conditions pointwise. Inspired by \cite{guzman2017finite}, we further add penalty terms on the flux $\partial^2_{n} u$ to enhance the numerical stability. These terms are applied to all edges of interface elements. A similar idea was employed in \cite{2021ChenZhang} where additional penalty terms were introduced to stabilize the pressure field in Stokes interface problems.

\subsection{Well-posedness of the Numerical Solution}
We establish the well-posedness of the numerical solution $u_h$ obtained by the scheme \eqref{equ:scheme}. Define the mesh-dependent norm $\|\cdot\|_h$ on $S_h^p$ by
\begin{equation}
	\begin{split}
		\|v\|_h^2=&\sum_{T\in \mathcal{T}_h} \beta\| \nabla^2v\|^2_{L^2(T)}+ \sum_{e\in \mathring{\mathcal{E}}^n_h}\frac{\sigma_u \beta}{|e|}\left\|\jump{\partial_n v}\right\|^2_{L^2(e)} + \sum_{T \in \mathcal{T}_h^i}\frac{\sigma_u \aver{\beta}}{h_T} \mynorm{\jump{\partial_n v}}^2_{L^2(\Gamma_T)}\\
		+& \sum_{e\in \mathring{\mathcal{E}}^i_h}\sigma_u\left(\frac{ \beta^+}{|e^+|}\mynorm{\jump{\partial_n v}}^2_{L^2(e^+)}+\frac{ \beta^-}{|e^-|}\mynorm{\jump{\partial_n v}}^2_{L^2(e^-)} \right) + \sum_{e\in \mathring{\mathcal{E}}_h^n \cap \mathcal{F}_h^i}\sigma_F |e|\beta \mynorm{\jump{\partial_{nn} v }}^2_{L^2(e)}\\
		+& \sum_{e\in \mathring{\mathcal{E}}^i_h}\sigma_F\left( |e^+|\beta^+ \mynorm{\jump{\partial_{nn} v }}^2_{L^2(e^+)} + |e^-|\beta^- \mynorm{\jump{\partial_{nn} v }}^2_{L^2(e^-)} \right) +\sum_{e\in \mathring{\mathcal{E}}^i_h}\frac{\sigma_n \aver{\beta}}{|e|^3 } \mynorm{ \jump{v}}^2_{L^2(e)}\\
		+& \sum_{T \in \mathcal{T}_h^i}\frac{\sigma_n \aver{\beta}}{h_T^3 } \mynorm{\jump{v}}^2_{L^2(\Gamma_T)}.
	\end{split}
\end{equation}

\begin{lemma}
	$\|\cdot\|_h$ defines a norm on $S_h^p$.
\end{lemma}
\begin{proof}
	It is clear that $\|\cdot\|_h$  defines a semi-norm, so it remains to show positivity. Let $v \in S_h^p$ and assume $\|v\|_h=0$. Then for each $T$ for $\mathcal{T}_h^n$ and each each subelement $T^{\pm}$ with $\mathcal{T}_h^i$, we have $\| \nabla^2v\|^2=0$. Hence, $v$ is piecewise linear. By the definition of $\|\cdot\|_h$, the jump of $v$ vanish across all non-interface edges, interface edges $e \in \mathcal{E}_h^i$, and interface segments $\Gamma_T$. Thus, $v$ is continuous across the entire domain $\Omega$. 
	 Similarly, the jumps of normal derivatives vanish, indicating $v$ is globally linear. Since $v|_{\partial \Omega}=0$ and $\Omega$ is a polygon domain, it follows that  $v \equiv 0$. Hence $\|\cdot\|_h$ is a norm on $S_h^p$.
\end{proof}

Next we prove the trace inequality for IFE spaces. We first recall the standard trace inequality for a $k$th-degree polynomial $v$ on a triangle $T$, with edge $e$ \cite{warburton2003constants}:
\begin{equation}
\|v\|_{L^{2}(e)} \leqslant \sqrt{\frac{(k+1)(k+2)}{2} \frac{|e|}{|T|}}\|v\|_{L^{2}(T)}.
\end{equation}
Building on this estimate, we derive trace inequalities tailored to IFE functions on interface elements/edges.
\begin{lemma}
	\label{lemma:weaktrace}
	Let $e\in \mathcal{E}_h^i$ be an interface edge with neighboring element $T_e^1$ and $T_e^2$. Suppose $v$ is a polynomial of degree $k$. Then for $s=+,-$, there exists a constant $C$, independent of interface location and $h$, such that the following estimate holds for at least one of $i\in\{1,2\}$:
	\begin{equation}
		\begin{split}
			\|v\|_{L^2(e^s)} \leq C|e^s|^{-1/2}\|v\|_{L^2(T_e^{i,s})}.
		\end{split}
	\end{equation}
\end{lemma}
\begin{proof}
	Without loss of generality, consider the isosceles subtriangle $\tilde{T}_1 \subset T_e^{1,s}$ with base edge $e^s$ and base angle $\tilde{\alpha}=\min\{\underline{\alpha},\pi/4\}$, as constructed in the proof of Lemma 6 in \cite{guzman2017finite} (see Figure \ref{fig:trace}). The existence of such a subtriangle is guaranteed by Lemma \ref{lemma:rt}. Here, $\underline{\alpha}$ is the minimum nonzero angle of the triangulation whose existence is ensured by the shape regularity assumption on $\mathcal{T}_h$. Clearly, $\tilde \alpha$ is independent of the interface location. Applying the standard trace inequality on $\tilde{T}_1$, we obtain
	\begin{equation}
			\|v\|_{L^{2}(e^s)} \leqslant \sqrt{\frac{(k+1)(k+2)}{2} \frac{|e^s|}{|\tilde{T}_1|}}\|v\|_{L^{2}(\tilde{T} _1)} = 
			\sqrt{\frac{2(k+1)(k+2)}{|e^s|\tan\tilde{\alpha}}}\|v\|_{L^{2}(T_e^{1,s})}.
	\end{equation}
	The equality is due to $|\tilde{T}_1| = \frac{1}{\YCrev{4}}|e^s|^2\tan(\tilde{ \alpha})$. Setting $C = \sqrt{\frac{2(k+1)(k+2)}{\tan\tilde{\alpha}}}$, we obtain the estimate. The constant $C$ depends only on the polynomial degree $k$ and shape-regularity constant.
\end{proof}

\begin{figure}[htbp]
 \centering
 \includegraphics[width=.5\textwidth]{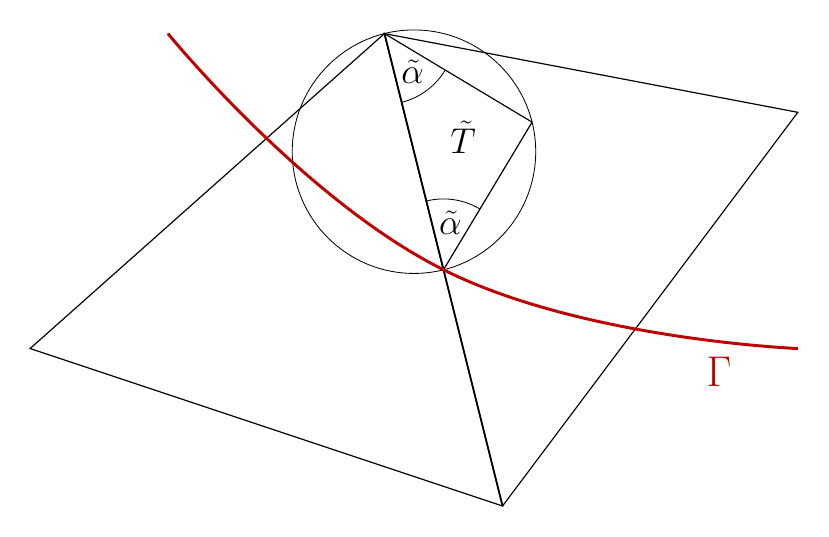}
 \caption{An illustration of triangular $\tilde T$}
 \label{fig:trace}
\end{figure}

Then we derive the following estimate for the symmetric term of \eqref{equ:scheme} on interface edge.
\begin{lemma}
	\label{lemma:symmbound}
	Let $u,v\in S_h^p$. Then
	\begin{equation}
		\begin{split}
			&\bigg| \sum_{e\in \mathring{\mathcal{E}}_h^i}\int_{e} \aver{\beta\partial_{nn} u }\jump{\partial_n v} d s \bigg| \\ 
			\leq &\frac{1}{8}\bigg(\sum_{T\in \mathcal{T}_h}  \beta\| \nabla^2u\|^2_{L^2(T)} +\sum_{e\in \mathring{\mathcal{E}}_h^i} \beta^+|e^+|\mynorm{\jump{\partial_{nn} u }}^2_{L^2(e^+)}+\sum_{e\in \mathring{\mathcal{E}}_h^i} \beta^-|e^-|\mynorm{\jump{\partial_{nn} u }}^2_{L^2(e^-)} \bigg)\\
			&+C\bigg(\sum_{e\in \mathring{\mathcal{E}}_h^i}\frac{\beta^+}{|e^+|}\mynorm{\jump{\partial_n v}}^2_{L^2(e^+)}+\sum_{e\in \mathring{\mathcal{E}}_h^i}\frac{\beta^-}{|e^-|}\mynorm{\jump{\partial_n v}}^2_{L^2(e^-)} \bigg),
		\end{split}	
	\end{equation}
where $C>0$ is independent of the interface location and the mesh size $h$. 
\end{lemma}

\begin{proof}
	On the subedge $e^+$, we have the following equality:
	\begin{equation}\label{eq: trace nn}
		\begin{split}
			 \sum_{e\in \mathring{\mathcal{E}}_h^i}\int_{e^+} \aver{\beta\partial_{nn} u }\jump{\partial_n v} d s 
		   = & \sum_{e\in \mathring{\mathcal{E}}_h^i}\int_{e^+} \bigg(\beta\partial_{nn} u  \big|_{T_e^2} + \frac{1}{2}\jump{\beta\partial_{nn} u }\bigg)\jump{\partial_n v} d s\\
		   = &  \sum_{e\in \mathring{\mathcal{E}}_h^i}\int_{e^+} \beta\partial_{nn} u  \big|_{T_e^2} \jump{\partial_n v} d s + \frac{1}{2}\sum_{e\in \mathring{\mathcal{E}}_h^i}\int_{e^+}  \jump{\beta\partial_{nn} u }\jump{\partial_n v} d s  =: Q_1 + Q_2 . \\
		\end{split}
	\end{equation}
	Without loss of generality, choose $T_e^{2,+}$ to be the subelement satisfying Lemma \ref{lemma:weaktrace}. Then by Cauchy-Schwarz inequality, we have
		\begin{equation}
		\begin{split}
			|Q_1| \leq \bigg(\sum_{e\in \mathring{\mathcal{E}}_h^i} \beta^+|e^+|\mynorm{\partial_{nn} u  \big|_{T_e^2}}^2_{L^2(e^+)} \bigg)^{\frac{1}{2}}\bigg(\sum_{e\in \mathring{\mathcal{E}}_h^i}\frac{\beta^+}{|e^+|}\mynorm{\jump{\partial_n v}}^2_{L^2(e^+)} \bigg)^{\frac{1}{2}}.
		\end{split}
	\end{equation}
	Applying Lemma \ref{lemma:weaktrace} on  $T_e^{2,+}$, we have
	\begin{equation}
		\label{equ:Q1bound}
		\begin{split}
			|Q_1|\leq & C\bigg(\sum_{e\in \mathring{\mathcal{E}}_h^i} \beta^+\mynorm{\partial_{nn} u }^2_{L^2(T_e^{2,+})} \bigg)^{\frac{1}{2}}\bigg(\sum_{e\in \mathring{\mathcal{E}}_h^i}\frac{\beta^+}{|e^+|}\mynorm{\jump{\partial_n v}}^2_{L^2(e^+)} \bigg)^{\frac{1}{2}}\\
				\leq & C\bigg(\sum_{T\in \mathcal{T}_h^i} \beta^+\| \nabla^2u\|^2_{L^2(T^+)} \bigg)^{\frac{1}{2}}\bigg(\sum_{e\in \mathring{\mathcal{E}}_h^i}\frac{\beta^+}{|e^+|}\mynorm{\jump{\partial_n v}}^2_{L^2(e^+)} \bigg)^{\frac{1}{2}}.
		\end{split}
	\end{equation}
	By Young's inequality $ab\le \frac{1}{8}a^2+2b^2$, we have
	\begin{equation}
		|Q_1| \leq \frac{1}{8}\bigg(\sum_{T\in \mathcal{T}_h^i} \beta^+\| \nabla^2u\|^2_{L^2(T^+)} \bigg)+2C^2\bigg(\sum_{e\in \mathring{\mathcal{E}}_h^i}\frac{\beta^+}{|e^+|}\mynorm{\jump{\partial_n v}}^2_{L^2(e^+)} \bigg).	
	\end{equation}
	For $Q_2$, Cauchy-Schwarz and Young's inequality give
	\begin{equation}
		\label{equ:Q2bound}
		\begin{split}
			|Q_2| &\leq \frac{1}{2} \bigg(\sum_{e\in \mathring{\mathcal{E}}_h^i} \beta^+|e^+|\mynorm{\jump{\partial_{nn} u }}^2_{L^2(e^+)} \bigg)^{\frac{1}{2}}\bigg(\sum_{e\in \mathring{\mathcal{E}}_h^i}\frac{\beta^+}{|e^+|}\mynorm{\jump{\partial_n v}}^2_{L^2(e^+)} \bigg)^{\frac{1}{2}}\\
			      &\leq \frac{1}{8} \bigg(\sum_{e\in \mathring{\mathcal{E}}_h^i} \beta^+|e^+|\mynorm{\jump{\partial_{nn} u }}^2_{L^2(e^+)} \bigg)+\frac{1}{2}\bigg(\sum_{e\in \mathring{\mathcal{E}}_h^i}\frac{\beta^+}{|e^+|}\mynorm{\jump{\partial_n v}}^2_{L^2(e^+)} \bigg).
		\end{split}
	\end{equation}
	Combine estimates of $Q_1$, $Q_2$, we have:
	\begin{equation}
		\begin{split}
			\bigg| \sum_{e\in \mathring{\mathcal{E}}_h^i}\int_{e^+} \aver{\beta\partial_{nn} u }\jump{\partial_n v} d s \bigg| 
			\leq & \frac{1}{8}\bigg(\sum_{T\in \mathcal{T}_h^i} \beta^+\| \nabla^2u\|^2_{L^2(T^+)} +\sum_{e\in \mathring{\mathcal{E}}_h^i} \beta^+|e^+|\mynorm{\jump{\partial_{nn} u }}^2_{L^2(e^+)} \bigg)\\
			&+(2C^2+\frac{1}{2})\bigg(\sum_{e\in \mathring{\mathcal{E}}_h^i}\frac{\beta^+}{|e^+|}\mynorm{\jump{\partial_n v}}^2_{L^2(e^+)} \bigg).
		\end{split}	
	\end{equation}
	An analogous estimate holds on $e^-$. Adding the two sides, we obtain the estimate \eqref{eq: trace nn}. 
\end{proof}

We have a similar estimate for the edges $e\in \mathring{\mathcal{E}}_h^n \cap \mathcal{F}_h^i$ neighbored by a non-interface element and interface element. This result is summarized in the following Lemma.
\begin{lemma}
	\label{lemma:symmbound2}
	Let $u,v\in S_h^p$. Then
	\begin{equation}\label{eq: lemma sym bound2}
		\begin{split}
			\bigg| \sum_{e\in \mathring{\mathcal{E}}_h^n \cap \mathcal{F}_h^i}\int_{e} \aver{\beta\partial_{nn} u }\jump{\partial_n v} d s \bigg|  
			\leq &\frac{1}{8}\bigg(\sum_{T\in \mathcal{T}_h}  \beta\| \nabla^2u\|^2_{L^2(T)} +\sum_{e\in \mathring{\mathcal{E}}_h^n \cap \mathcal{F}_h^i} \beta|e|\mynorm{\jump{\partial_{nn} u }}^2_{L^2(e)}\bigg)\\
			&+C\sum_{e\in \mathring{\mathcal{E}}_h^n}\frac{\beta}{|e|}\mynorm{\jump{\partial_n v}}^2_{L^2(e)}.
		\end{split}	
	\end{equation}
	where $C>0$ is independent of the interface location and the mesh size $h$.
\end{lemma}
\begin{proof}
	For $e\in \mathring{\mathcal{E}}_h^n \cap \mathcal{F}_h^i$, at least one neighborhood element is a non-interface element. Without loss of generality, assume $T_e^2$ is non-interface. Then
	\begin{equation}
		\begin{split}
			 &\sum_{e\in \mathring{\mathcal{E}}_h^n \cap \mathcal{F}_h^i}\int_{e} \aver{\beta\partial_{nn} u }\jump{\partial_n v} d s 
			 = \sum_{e\in \mathring{\mathcal{E}}_h^n \cap \mathcal{F}_h^i}\int_{e} \bigg(\beta\partial_{nn} u  \big|_{T_e^2} \bigg)\jump{\partial_n v} d s + \frac{1}{2}\sum_{e\in \mathring{\mathcal{E}}_h^n \cap \mathcal{F}_h^i}\int_{e}  \jump{\beta\partial_{nn} u }\jump{\partial_n v} d s =:Q_3+Q_4.\\
		\end{split}
	\end{equation}
	For $Q_3$, applying standard trace inequality similar to Lemma \ref{lemma:weaktrace} and summing over all elements gives
	\begin{equation}\label{equ:Q3bound}
	|Q_3|\leq  C\bigg(\sum_{T\in \mathcal{T}_h} \beta\| \nabla^2u\|^2_{L^2(T)} \bigg)^{\frac{1}{2}}\bigg(\sum_{e\in \mathring{\mathcal{E}}_h^n \cap \mathcal{F}_h^i}\frac{\beta}{|e|}\mynorm{\jump{\partial_n v}}^2_{L^2(e)} \bigg)^{\frac{1}{2}}.
	\end{equation}
	For $Q_4$, Cauchy-Schwarz inequality gives 
	\begin{equation}
		\label{equ:Q4bound}
			|Q_4| \leq  \frac{1}{2} \bigg(\sum_{e\in \mathring{\mathcal{E}}_h^n \cap \mathcal{F}_h^i} \beta|e|\mynorm{\jump{\partial_{nn} u }}^2_{L^2(e)} \bigg)^{\frac{1}{2}}\bigg(\sum_{e\in \mathring{\mathcal{E}}_h^n \cap \mathcal{F}_h^i}\frac{\beta}{|e|}\mynorm{\jump{\partial_n v}}^2_{L^2(e)} \bigg)^{\frac{1}{2}}.
	\end{equation}
	Applying Young's inequality to the bounds for $Q_3$ and $Q_4$ yields the estimate \eqref{eq: lemma sym bound2}.
\end{proof}

Next we  prove the continuity and coercivity for $a_h(u,v)$:
\begin{theorem}[Continuity]
	For all $u, v \in S_h^p$, there exists a constant $C$ independent with the interface location such that:
	\begin{equation}\label{eq: continuity}
		a_h(u,v) \leq C\mynorm{u}_h\mynorm{v}_h.
	\end{equation}
\end{theorem}
\begin{proof}
	It suffices to bound the symmetric terms, since the other terms follow directly by Cauchy-Schwarz inequality.
	For $e\in \mathcal{E}_h^n \backslash \mathcal{F}_h^i$, the standard estimate \cite{brenner2011c} together with trace and inverse inequalities gives
	\begin{equation}
		\bigg|\sum_{e\in \mathcal{E}_h^n \backslash \mathcal{F}_h^i} \int_e \aver{\beta\partial_{nn} u} \jump{\partial_n v}ds \bigg| \leq C \bigg(\sum_{T\in \mathcal{T}_h} \beta\| \nabla^2u\|^2_{L^2(T)}\bigg)^{\frac{1}{2}} \bigg( \sum_{e\in \mathring{\mathcal{E}}^n_h}\frac{\beta}{|e|}\left\|\jump{\partial_n v}\right\|^2 \bigg)^{\frac{1}{2}}\leq C \mynorm{u}_h\mynorm{v}_h.
	\end{equation}
	For $e\in \mathcal{E}_h^n \cap \mathcal{F}_h^i$, the bound follows from \eqref{equ:Q3bound} and \eqref{equ:Q4bound} . For $e \in \mathring{\mathcal{E}_h^i}$ we use \eqref{equ:Q1bound} and \eqref{equ:Q2bound}. 
	The stabilization terms in $J_{h,u}(u,v)$ and $J_{h,n}(u,v)$ are bounded by Cauchy-Schwarz inequality. Finally, for $A_h(u,v)$ we have
	\begin{equation}
		\bigg| A_h(u,v) \bigg|	 \leq \bigg(\sum_{T\in \mathcal{T}_h} \beta\| \nabla^2u\|^2_{L^2(T)} \bigg)^{\frac{1}{2}} \bigg(\sum_{T\in \mathcal{T}_h} \beta\| \nabla^2v\|^2_{L^2(T)} \bigg)^{\frac{1}{2}} \leq \mynorm{u}_h\mynorm{v}_h.
	\end{equation}
	Combining all estimates gives the continuous result \eqref{eq: continuity}.
\end{proof}

\begin{theorem}[Coercivity]
	For all $v \in S_h^p$ {, $\sigma_u$ sufficient large and $\sigma_F>0$,} there exists a constant $c$,, independent of the interface location, such that
	\begin{equation}
		a_h(v,v) \geq c \mynorm{v}^2_h.
	\end{equation}
\end{theorem}
\begin{proof}
	As before, we only need to estimate the symmetric terms. For $e\in \mathring{\mathcal{E}}_h^n \backslash \mathcal{F}_h^i$, the standard estimate yields
	\begin{equation}
		\begin{split}
			- \sum_{e\in \mathcal{E}_h^n \backslash \mathcal{F}_h^i}\int_{e} \aver{\beta\partial_{nn} v }\jump{\partial_n v} d s	 \geq & -C \bigg(\sum_{T\in \mathcal{T}_h}  \beta\| \nabla^2u\|^2_{L^2(T)}\bigg)^{\frac{1}{2}}\bigg( \sum_{e\in \mathring{\mathcal{E}}^n_h}\frac{\beta}{|e|}\left\|\jump{\partial_n v}\right\|^2 \bigg)^{\frac{1}{2}}\\
			\geq & -C\bigg[ \frac{1}{8C} \bigg(\sum_{T\in \mathcal{T}_h} \beta\| \nabla^2u\|^2_{L^2(T)}\bigg)+2C\bigg( \sum_{e\in \mathring{\mathcal{E}}^n_h}\frac{2\aver{\beta}_e}{|e|}\left\|\jump{\partial_n v}\right\|^2\bigg) \bigg]\\
		 \geq& -\frac{1}{8} \bigg(\sum_{T\in \mathcal{T}_h}  \beta\| \nabla^2u\|^2_{L^2(T)}\bigg)-C\bigg( \sum_{e\in \mathring{\mathcal{E}}^n_h}\frac{\aver{\beta}_e}{|e|}\left\|\jump{\partial_n v}\right\|^2 \bigg).
		\end{split}
	\end{equation}
	For $e\in \mathcal{E}_h^n \cap \mathcal{F}_h^i$, Lemma \ref{lemma:symmbound2} gives
	\begin{equation}
		\begin{split}
			&-\sum_{e\in \mathring{\mathcal{E}}_h^n \cap \mathcal{F}_h^i}\int_{e} \aver{\beta\partial_{nn} v }\jump{\partial_n v} d s  \\ 
			\geq &-\frac{1}{8}\bigg(\sum_{T\in \mathcal{T}_h} \beta\| \nabla^2v\|^2_{L^2(T)} +\sum_{e\in \mathring{\mathcal{E}}_h^n \cap \mathcal{F}_h^i} \beta|e|\mynorm{\jump{\partial_{nn} v }}^2_{L^2(e)}\bigg)-C\bigg(\sum_{e\in \mathring{\mathcal{E}}_h^n}\frac{\beta}{|e|}\mynorm{\jump{\partial_n v}}^2_{L^2(e)}\bigg).
		\end{split}	
	\end{equation}
	For $e\in \mathring{\mathcal{E}}_h^i$, we use Lemma \ref{lemma:symmbound} to get:
	\begin{equation}
		\begin{split}
			&-\sum_{e\in \mathring{\mathcal{E}}_h^i}\int_{e} \aver{\beta\partial_{nn} v }\jump{\partial_n v} d s  \\ 
			\geq & -\frac{1}{8}\bigg(\sum_{T\in \mathcal{T}_h}  \beta\| \nabla^2v\|^2_{L^2(T)} +\sum_{e\in \mathring{\mathcal{E}}_h^i} \beta^+|e^+|\mynorm{\jump{\partial_{nn} v }}^2_{L^2(e^+)}+\sum_{e\in \mathring{\mathcal{E}}_h^i} \beta^-|e^-|\mynorm{\jump{\partial_{nn} v }}^2_{L^2(e^-)} \bigg)\\
			&-C\bigg(\sum_{e\in \mathring{\mathcal{E}}_h^i}\frac{\beta^+}{|e^+|}\mynorm{\jump{\partial_n v}}^2_{L^2(e^+)}+\sum_{e\in \mathring{\mathcal{E}}_h^i}\frac{\beta^-}{|e^-|}\mynorm{\jump{\partial_n v}}^2_{L^2(e^-)}\bigg).
		\end{split}	
	\end{equation}
	Taking the generic $C$ to be the maximum in the above three estimate, we have
	\begin{equation}
		\begin{split}
			a_h(v,v) \geq &\frac{1}{4} \sum_{T\in \mathcal{T}_h}  \beta\| \nabla^2v\|^2_{L^2(T)} +(\sigma_u-2C) \sum_{e\in \mathring{\mathcal{E}}^n_h}\frac{\aver{\beta}_e}{|e|}\left\|\jump{\partial_n v}\right\|^2+ \sum_{T \in \mathcal{T}_h^i}\frac{\sigma_u \aver{\beta}_e}{|h_T|} \mynorm{\jump{\partial_n v}}^2_{L^2(\Gamma_T)}\\
			+&(\sigma_u-C)\bigg(\sum_{e\in \mathring{\mathcal{E}}_h^i}\frac{\beta^+}{|e^+|}\mynorm{\jump{\partial_n v}}^2_{L^2(e^+)}+\sum_{e\in \mathring{\mathcal{E}}_h^i}\frac{\beta^-}{|e^-|}\mynorm{\jump{\partial_n v}}^2_{L^2(e^-)} \bigg) \\
			+&(\sigma_F-\frac{1}{8}) \bigg( \sum_{e\in \mathring{\mathcal{E}}_h^n \cap \mathcal{F}_h^i} |e|\beta \mynorm{\jump{\partial_{nn} v }}^2_{L^2(e)}+\sum_{e\in \mathring{\mathcal{E}}_h^i} \beta^+|e^+|\mynorm{\jump{\partial_{nn} v }}^2_{L^2(e^+)}+\sum_{e\in \mathring{\mathcal{E}}_h^i} \beta^-|e^-|\mynorm{\jump{\partial_{nn} v }}^2_{L^2(e^-)} \bigg)\\
			+&\sum_{e\in \mathring{\mathcal{E}}^i_h}\frac{\sigma_n \aver{\beta}_e}{|e|^3 } \mynorm{ \jump{v}}^2_{L^2(e)}+
		 \sum_{T \in \mathcal{T}_h^i}\frac{\sigma_n \aver{\beta}_e}{h_T^3 } \mynorm{\jump{v}}^2_{L^2(\Gamma_T)}.
		\end{split}
	\end{equation}
	Choosing $\sigma_u=3C$, $\sigma_F=1$, and defining With $c=\min\{\frac{1}{4},2C\}$, we have the coercivity result.
\end{proof}
The well-posedness of our numerical scheme follows directly from the continuity and coercivity results, and the Lax-Milgram Theorem.


\section{Numerical Examples}
\label{sec:experiment}
In this section, we present three numerical examples to demonstrate the performance of the immersed $C^0$ interior penalty method. Throughout all computation, the computational domain is chosen as $\Omega = [-1,1]\times[-1,1]$. 

The interface-unfitted triangular mesh $\mathcal{T}_h$ is constructed as follows:  we first partition $\Omega$ into $N^2$ uniform squares with side length $h=2/N$, and then subdivide each square into two triangles by its diagonals with the negative slope. 

For the construction of the IFE shape functions, we set the parameters in the bilinear form \eqref{equ:bilinearform2} as 
 $\omega_0=\max\{\beta^+,\beta^-\}^2$ and $\omega_1=\omega_2=\omega_3=1$. Numerical quadratures on curve interface segments and curved subelements are carried out by suitable mappings onto line segment and standard triangles, respectively.

\begin{figure}[htb]
 \centering
 \includegraphics[width=.9\textwidth]{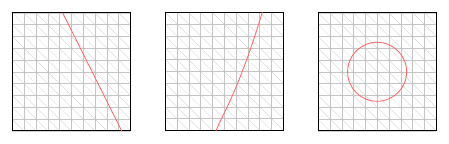}
 \caption{Illustration of interfaces used in Examples}
 \label{fig:interface}
\end{figure}

\subsection{Straight Line Interface}
\label{sec:line}
In this first example, we test the accuracy and consistency of proposed method with a straight line interface. The interface is represented by the level set
\[\Gamma=\{(x,y)\in \Omega:2x+y-c=0\},\]
which partitions the domain into $\Omega^+=\{(x,y)\in \Omega:2x+y-c>0\}$ and $\Omega^-=\{(x,y)\in \Omega:2x+y-c<0\}$. The exact solution is given by 
\begin{equation}
	\label{equ:lineex}
	u(x, y)=\left\{\begin{array}{ll}
	\frac{1}{\beta^-}{(2x+y-c)^2 \sin^2(\pi y)}, & x \in \Omega^{-} \\
	\frac{1}{\beta^+}{(2x+y-c)^2 \sin^2(\pi y)}, & x \in \Omega^{+}.
	\end{array}\right.
\end{equation}
In the following result, we take $c=\sqrt{0.5}$, see left plot of Figure \ref{fig:interface} for a visualization.

Tables \ref{tab:line_itp_p2} and \ref{tab:line_itp_p3} report the  interpolation errors of the $\mathcal{P}_2$ and $\mathcal{P}_3$ IFE spaces for the case $(\beta^-,\beta^+)=(1,100)$. Both spaces exhibit optimal approximation order of convergence:
\begin{equation}
\|\mathcal{I}_h^p u - u\|_0 + h|\mathcal{I}_h^p u - u|_1 + h^2|\mathcal{I}_h^p u - u|_2 \approx \mathcal{O}(h^{p+1}),~~~p=2,3. 
\end{equation}

Tables \ref{tab:line_sol_p2} and Table \ref{tab:line_sol_p3} display the numerical solution errors. For $p=2$, we observe optimal convergence in the $H^1$ and $H^2$-seminorms, and suboptimal convergence in $L^2$-norm, consistent with the known behavior of the $C^0$ interior penalty method for non-interface problems \cite{engel2002continuous}. For $p=3$, all norms convergence optimally, These results confirm that our immersed finite element spaces, combined with the proposed immersed $C^0$ interior penalty method, achieves the expected accuracy for the line interface case. 

\begin{table}[htbp]
\centering
\resizebox{.8\textwidth}{!}{
\begin{tabular}{ccccccc}
	\toprule
	$N$ & $||\mathcal{I}^p_hu-u||_{0}$ & order&$|\mathcal{I}^p_hu-u|_{1}$&order&$|\mathcal{I}^p_hu-u|_{2}$&order\\
	\midrule
	$10$&$2.2981\times 10^{-2}$&    &$9.2168\times 10^{-1}$&    &$3.5313\times 10^{1}$&\\
	$20$&$2.9230\times 10^{-3}$&2.97&$2.3391\times 10^{-1}$&1.98&$1.7926\times 10^{1}$&0.98\\
	$40$&$3.6703\times 10^{-4}$&2.99&$5.8708\times 10^{-2}$&1.99&$8.9983\times 10^{0}$&0.99\\
	$80$&$4.5932\times 10^{-5}$&3.00&$1.4692\times 10^{-2}$&2.00&$4.5036\times 10^{0}$&1.00\\
	$160$&$5.7431\times 10^{-6}$&3.00&$3.6738\times 10^{-3}$&2.00&$2.2524\times 10^{0}$&1.00\\
	\bottomrule
\end{tabular}
}
	\caption{Interpolation error of immersed $\mathcal{P}_2$ element of line interface when $\beta^+=100$, $\beta^-=1$}
	\label{tab:line_itp_p2}
\end{table}

\begin{table}[htbp]
\centering
\resizebox{.8\textwidth}{!}{
\begin{tabular}{ccccccc}
	\toprule
	$N$ & $||u_h-u||_{0}$ & order&$|u_h-u|_{1}$&order&$|u_h-u|_{2}$&order\\
	\midrule
	$10$&$1.1367\times 10^{-1}$&    &$1.0854\times 10^{0}$&    &$3.5853\times 10^{1}$&\\
	$20$&$2.7870\times 10^{-2}$&2.03&$2.6715\times 10^{-1}$&2.02&$1.8030\times 10^{1}$&0.99\\
	$40$&$7.2460\times 10^{-3}$&1.94&$6.7052\times 10^{-2}$&1.99&$9.0140\times 10^{0}$&1.00\\
	$80$&$1.8531\times 10^{-3}$&1.97&$1.6832\times 10^{-2}$&1.99&$4.5059\times 10^{0}$&1.00\\
	$160$&$4.6735\times 10^{-4}$&1.99&$4.2162\times 10^{-3}$&2.00&$2.2527\times 10^{0}$&1.00\\
	\bottomrule
\end{tabular}
}
	\caption{Numerical solution errors of $\mathcal{P}_2$ element of line interface when $\beta^+=100$, $\beta^-=1$}
	\label{tab:line_sol_p2}
\end{table}

\begin{table}[htbp]
\centering
\resizebox{.8\textwidth}{!}{
\begin{tabular}{ccccccc}
	\toprule
	$N$ & $||\mathcal{I}^p_hu-u||_{0}$ & order&$|\mathcal{I}^p_hu-u|_{1}$&order&$|\mathcal{I}^p_hu-u|_{2}$&order\\
	\midrule
	$10$&$1.6860\times 10^{-3}$&    &$1.0195\times 10^{-1}$&    &$6.4074\times 10^{0}$&\\
	$20$&$1.0708\times 10^{-4}$&3.98&$1.2934\times 10^{-2}$&2.98&$1.6289\times 10^{0}$&1.98\\
	$40$&$6.7182\times 10^{-6}$&3.99&$1.6225\times 10^{-3}$&2.99&$4.0884\times 10^{-1}$&1.99\\
	$60$&$1.3280\times 10^{-6}$&4.00&$4.8104\times 10^{-4}$&3.00&$1.8184\times 10^{-1}$&2.00\\
	$80$&$4.2028\times 10^{-7}$&4.00&$2.0298\times 10^{-4}$&3.00&$1.0231\times 10^{-1}$&2.00\\
	$100$&$1.7218\times 10^{-7}$&4.00&$1.0394\times 10^{-4}$&3.00&$6.5485\times 10^{-2}$&2.00\\
	\bottomrule
\end{tabular}
}
	\caption{Interpolation error of immersed $\mathcal{P}_3$ element of line interface when $\beta^+=100$, $\beta^-=1$}
	\label{tab:line_itp_p3}
\end{table}

\begin{table}[htbp]
\centering
\resizebox{.8\textwidth}{!}{
\begin{tabular}{ccccccc}
	\toprule
	$N$ & $||u_h-u||_{0}$ & order&$|u_h-u|_{1}$&order&$|u_h-u|_{2}$&order\\
	\midrule
	$10$&$6.4599\times 10^{-3}$&    &$1.3698\times 10^{-1}$&    &$6.0325\times 10^{0}$&\\
	$20$&$4.1086\times 10^{-4}$&3.97&$1.7319\times 10^{-2}$&2.98&$1.5031\times 10^{0}$&2.00\\
	$40$&$2.7180\times 10^{-5}$&3.92&$2.1956\times 10^{-3}$&2.98&$3.7432\times 10^{-1}$&2.01\\
	$60$&$5.4620\times 10^{-6}$&3.96&$6.5410\times 10^{-4}$&2.99&$1.6610\times 10^{-1}$&2.00\\
	$80$&$1.7460\times 10^{-6}$&3.96&$2.7676\times 10^{-4}$&2.99&$9.3358\times 10^{-2}$&2.00\\
	$100$&$7.1692\times 10^{-7}$&3.99&$1.4196\times 10^{-4}$&2.99&$5.9720\times 10^{-2}$&2.00\\
	\bottomrule
\end{tabular}
}
	\caption{Numerical solution errors of $\mathcal{P}_3$ element of line interface when $\beta^+=100$, $\beta^-=1$}
	\label{tab:line_sol_p3}
\end{table}

To further test the consistency with respect to the interface location, we consider a vertical interface aligned with $y$-axis. The new line interface can be written as
\[\Gamma=\{(x,y):x-c=0\}.\] 
The exact solution is 
\begin{equation}
	\label{equ:lineex2}
	u(x, y)=\left\{\begin{array}{ll}
	\frac{1}{\beta^-}{(x-c)^2 \sin^2(\pi y)}, & x\leq c \\
	\frac{1}{\beta^+}{(x-c)^2 \sin^2(\pi y)}, & x>c
	\end{array}\right.
\end{equation}

In this experiment, we set $(\beta^-,\beta^+)=(1,10)$ and fix the mesh size $h=2/40=0.05$. The parameter $c$ varies across the interval $[0.71,0.79]$, so that the interface $\{x=c\}$ moves through grid columns. When $c=0.75$, the interface align exactly with the mesh, and the immersed $C^0$ interior penalty method reduces to a standard $C^0$ interior penalty method. Figure \ref{fig:moving} reports the error in three norms as the interface moves. The results demonstrates that the errors evolve smoothly with respect to the interface location, including at the degenerated case when the interface coincides with the mesh. This confirms the robustness and consistency of our method under interface movement.

\begin{figure}[htb]
 \centering
 \includegraphics[width=.7\textwidth]{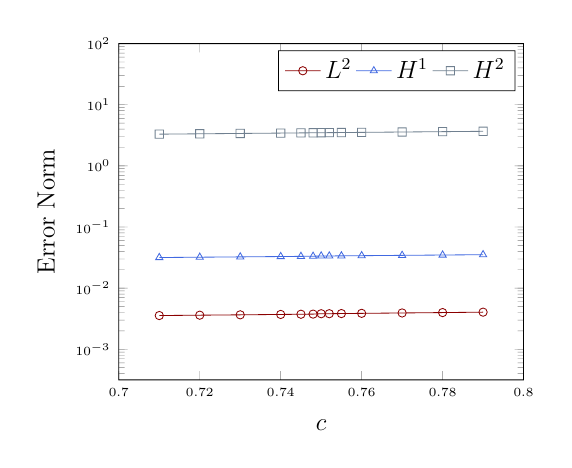}
 \caption{Errors of an moving line with respect to position}
 \label{fig:moving}
\end{figure}

\subsection{Parabola Interface Case}
\label{sec:parabola}
In this example, we test our method on a curved interface defined by a parabola. The interface is given by
\[\Gamma=\{(x,y)\in \Omega :y-(x^2+2x+c)=0\}\]
which divides the domain into $\Omega^+ = \{(x,y)\in \Omega :y-(x^2+2x+c)>0\}$ and $\Omega^- = \{(x,y)\in \Omega :y-(x^2+2x+c)<0\}$. We set $c=-\frac{\sqrt{2}}{2}$ in the experiment, see the middle plot of Figure \ref{fig:interface} for reference. The exact solution is defined by
\begin{equation}
	\label{equ:paraexample}
	u(x, y)=\left\{\begin{array}{ll}
	\frac{1}{\beta^{-}}{\left(x^{2}+2x+c-y\right)^{2}\left(1-y^2\right)^2}, & x \in \Omega^{-}, \\
	\frac{1}{\beta^{+}}{\left(x^{2}+2x+c-y\right)^{2}\left(1-y^2\right)^2}, & x \in \Omega^{+}.
	\end{array}\right.
\end{equation}

Tables \ref{tab:para_itp_p2} and \ref{tab:para_itp_p3} report interpolation errors of the immersed $\mathcal{P}_2$ and $\mathcal{P}_3$ spaces for the case $\beta^+=10$ and $\beta^-=1$. In both cases, we observe optimal convergence rates consistent with theoretical expectations.

\begin{table}[htbp]
\centering
\resizebox{.8\textwidth}{!}{
\begin{tabular}{ccccccc}
	\toprule
	$N$ & $||\mathcal{I}^p_hu-u||_{0}$ & order&$|\mathcal{I}^p_hu-u|_{1}$&order&$|\mathcal{I}^p_hu-u|_{2}$&order\\
	\midrule
	$10$&$5.8973\times 10^{-3}$&    &$2.5587\times 10^{-1}$&    &$8.2234\times 10^{0}$&\\
	$20$&$7.5802\times 10^{-4}$&2.96&$6.6025\times 10^{-2}$&1.95&$4.2167\times 10^{0}$&0.96\\
	$40$&$9.5420\times 10^{-5}$&2.99&$1.6647\times 10^{-2}$&1.99&$2.1224\times 10^{0}$&0.99\\
	$80$&$1.1949\times 10^{-5}$&3.00&$4.1718\times 10^{-3}$&2.00&$1.0631\times 10^{0}$&1.00\\
	$160$&$1.4944\times 10^{-6}$&3.00&$1.0437\times 10^{-3}$&2.00&$5.3187\times 10^{-1}$&1.00\\
	\bottomrule
\end{tabular}
}
	\caption{Interpolation error of immersed $\mathcal{P}_2$ element of parabola interface when when $\beta^+=10$, $\beta^-=1$}
	\label{tab:para_itp_p2}
\end{table}


\begin{table}[htb]
\centering
\resizebox{.8\textwidth}{!}{
\begin{tabular}{ccccccc}
	\toprule
	$N$ & $||\mathcal{I}^p_hu-u||_{0}$ & order&$|\mathcal{I}^p_hu-u|_{1}$&order&$|\mathcal{I}^p_hu-u|_{2}$&order\\
	\midrule
	$10$&$3.9532\times 10^{-4}$&    &$2.5529\times 10^{-2}$&    &$1.3731\times 10^{0}$&\\
	$20$&$2.5089\times 10^{-5}$&3.98&$3.2559\times 10^{-3}$&2.97&$3.4956\times 10^{-1}$&1.97\\
	$40$&$1.5807\times 10^{-6}$&3.99&$4.1044\times 10^{-4}$&2.99&$8.8171\times 10^{-2}$&1.99\\
	$60$&$3.1153\times 10^{-7}$&4.01&$1.2152\times 10^{-4}$&3.00&$3.9124\times 10^{-2}$&2.00\\
	$80$&$9.8835\times 10^{-8}$&3.99&$5.1370\times 10^{-5}$&2.99&$2.2057\times 10^{-2}$&1.99\\
	$100$&$4.0438\times 10^{-8}$&4.00&$2.6281\times 10^{-5}$&3.00&$1.4101\times 10^{-2}$&2.00\\
	\bottomrule
\end{tabular}
}
	\caption{Interpolation error of immersed $\mathcal{P}_3$ element of parabola interface when when $\beta^+=10$, $\beta^-=1$}
	\label{tab:para_itp_p3}
\end{table}


Figure \ref{fig:num_para} summarizes numerical solution errors for the parabola case with fixed $\beta^-=1$ and varying $\beta^+$. When $\beta^+ =1$, the interface disappears and the method reduces to the standard $C^0$ interior penalty method. In all cases, the proposed method achieves expected rates of convergence for both $\mathcal{P}_2$ and $\mathcal{P}_3$ spaces.
The results also indicates mild variations in error magnitudes across different parameter values., but the convergence remains robust. 
%

\begin{figure}[htb]
 \centering
 \includegraphics[width=.99\textwidth]{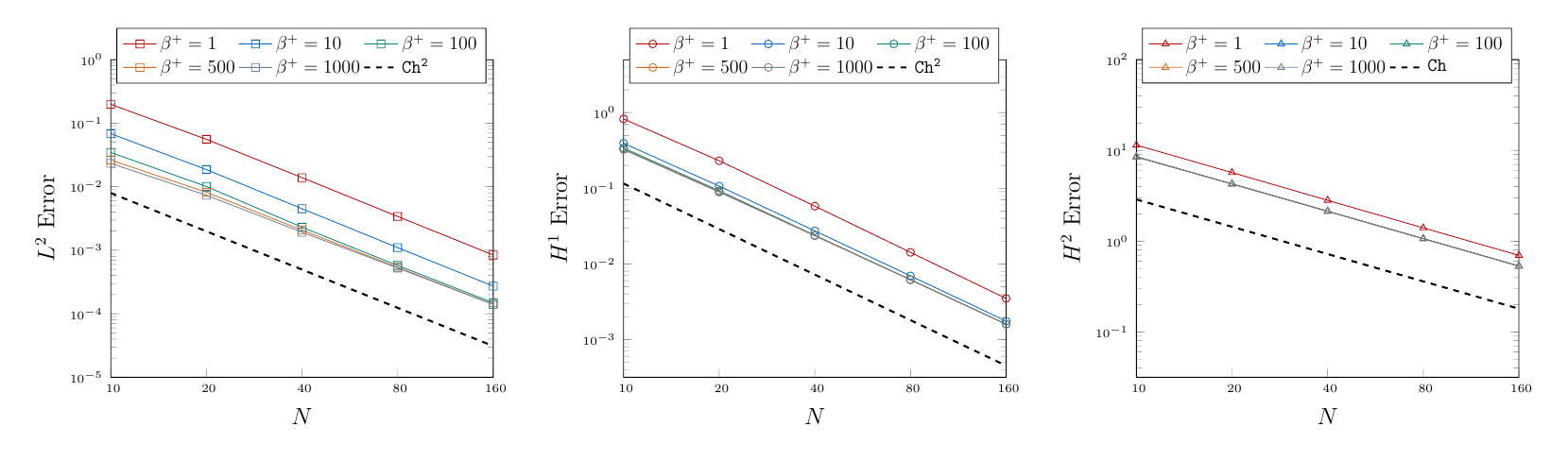}
 \includegraphics[width=.99\textwidth]{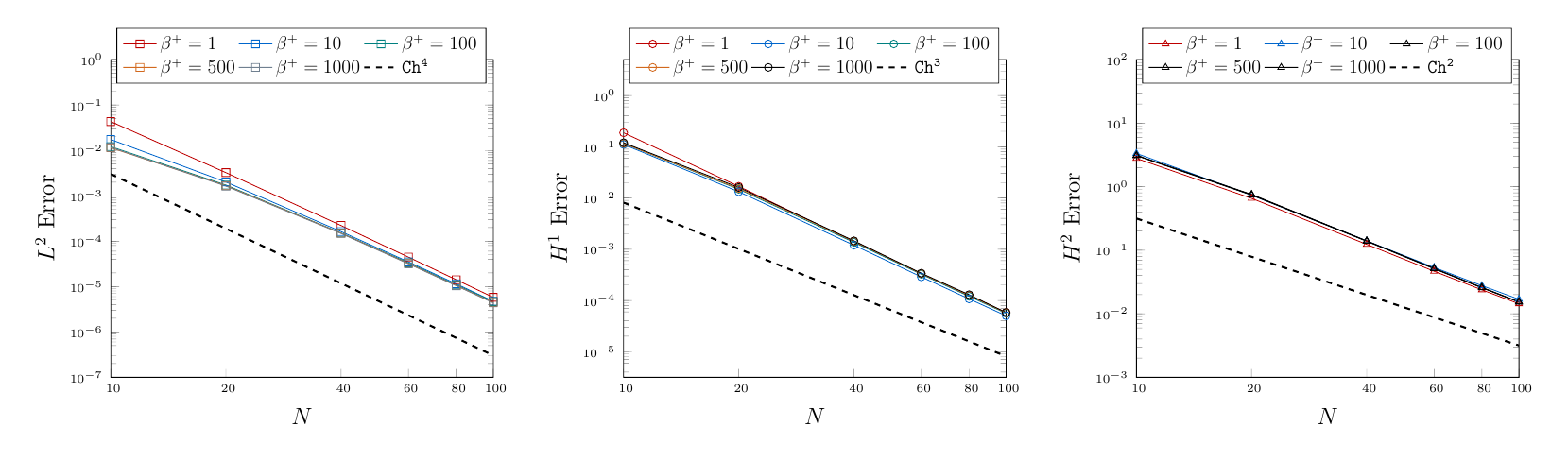}
 \caption{Numerical solution errors of $\mathcal{P}_2$ (first row) and $\mathcal{P}_3$ (second row)  for parabola interface case}
 \label{fig:num_para}
\end{figure}

\subsection{Circular Interface Case}
\label{sec:circle}
In this example, we consider a circular interface. The interface is the level-set function
\[\Gamma=\{(x,y)\in\Omega:x^2+y^2=r_0^2\}.\] 
which separates the domain into 
$\Omega^+=\{(x,y)\in\Omega:x^2+y^2>r_0^2\}$ and $\Omega^-=\{(x,y)\in\Omega:x^2+y^2<r_0^2\}$. The exact solution is 
\begin{equation}
	\label{equ:cirexample}
	u(x, y)=\left\{\begin{array}{ll}
	\frac{\left(x^{2}+y^{2}-r_{0}^{2}\right)^{2}\sin^2(\pi y)}{\beta^{-}}, & x \in \Omega^{-} \\
	\frac{\left(x^{2}+y^{2}-r_{0}^{2}\right)^{2}\sin^2(\pi y)}{\beta^{+}}, & x \in \Omega^{+}
	\end{array}\right.
\end{equation} 
We take $r_0=\pi/6.28$ in our tests, see the right plot of Figure \ref{fig:interface} for reference.

Tables \ref{tab:circle_itp_p2} and \ref{tab:circle_itp_p3} report interpolation errors for $\beta^+=1$ and $\beta^-=50$. As in previous examples, the immersed $\mathcal{P}_2$ and $\mathcal{P}_3$ spaces exhibit optimal approximation in $\|\cdot\|_0$, $|\cdot|_1$ and $|\cdot|_2$ norms. Tables \ref{tab:circle_sol_p2} and \ref{tab:circle_sol_p3} present the numerical solution errors for the proposed scheme. 
 
\begin{table}[htb]
\centering
\resizebox{.8\textwidth}{!}{
\begin{tabular}{ccccccc}
	\toprule
	$N$ & $||\mathcal{I}^p_hu-u||_{0}$ & order&$|\mathcal{I}^p_hu-u|_{1}$&order&$|\mathcal{I}^p_hu-u|_{2}$&order\\
	\midrule
	$10$&$7.5156\times 10^{-3}$&    &$3.0133\times 10^{-1}$&    &$1.1212\times 10^{1}$&\\
	$20$&$9.8140\times 10^{-4}$&2.94&$7.8906\times 10^{-2}$&1.93&$5.8092\times 10^{0}$&0.95\\
	$40$&$1.2383\times 10^{-4}$&2.99&$1.9919\times 10^{-2}$&1.99&$2.9275\times 10^{0}$&0.99\\
	$80$&$1.5514\times 10^{-5}$&3.00&$4.9916\times 10^{-3}$&2.00&$1.4666\times 10^{0}$&1.00\\
	$160$&$1.9404\times 10^{-6}$&3.00&$1.2486\times 10^{-3}$&2.00&$7.3364\times 10^{-1}$&1.00\\
	\bottomrule
\end{tabular}
}
	\caption{Interpolation error of immersed $\mathcal{P}_2$ element of circular interface when when $\beta^+=1$, $\beta^-=50$}
	\label{tab:circle_itp_p2}
\end{table}

\begin{table}[htb]
\centering
\resizebox{.8\textwidth}{!}{
\begin{tabular}{ccccccc}
	\toprule
	$N$ & $||u_h-u||_{0}$ & order&$|u_h-u|_{1}$&order&$|u_h-u|_{2}$&order\\
	\midrule
	$10$&$3.7308\times 10^{-2}$&    &$4.3174\times 10^{-1}$&    &$1.1364\times 10^{1}$&\\
	$20$&$1.0019\times 10^{-2}$&1.90&$1.2319\times 10^{-1}$&1.81&$5.8557\times 10^{0}$&0.96\\
	$40$&$2.9634\times 10^{-3}$&1.76&$3.4468\times 10^{-2}$&1.84&$2.9410\times 10^{0}$&0.99\\
	$80$&$8.2261\times 10^{-4}$&1.85&$9.1887\times 10^{-3}$&1.91&$1.4702\times 10^{0}$&1.00\\
	$160$&$2.1543\times 10^{-4}$&1.93&$2.3678\times 10^{-3}$&1.96&$7.3458\times 10^{-1}$&1.00\\
	\bottomrule
\end{tabular}
}
	\caption{Numerical solution errors of $\mathcal{P}_2$ element of circular interface when $\beta^+=1$, $\beta^-=50$}
	\label{tab:circle_sol_p2}
\end{table}

\begin{table}[htb]
\centering
\resizebox{.8\textwidth}{!}{
\begin{tabular}{ccccccc}
	\toprule
	$N$ & $||\mathcal{I}^p_hu-u||_{0}$ & order&$|\mathcal{I}^p_hu-u|_{1}$&order&$|\mathcal{I}^p_hu-u|_{2}$&order\\
	\midrule
	$10$&$7.0400\times 10^{-4}$&    &$4.3002\times 10^{-2}$&    &$2.5482\times 10^{0}$&\\
	$20$&$4.3799\times 10^{-5}$&4.01&$5.3550\times 10^{-3}$&3.01&$6.3730\times 10^{-1}$&2.00\\
	$40$&$2.7397\times 10^{-6}$&4.00&$6.6977\times 10^{-4}$&3.00&$1.5955\times 10^{-1}$&2.00\\
	$60$&$5.4131\times 10^{-7}$&4.00&$1.9849\times 10^{-4}$&3.00&$7.0938\times 10^{-2}$&2.00\\
	$80$&$1.7129\times 10^{-7}$&4.00&$8.3747\times 10^{-5}$&3.00&$3.9910\times 10^{-2}$&2.00\\
	$100$&$7.0169\times 10^{-8}$&4.00&$4.2880\times 10^{-5}$&3.00&$2.5545\times 10^{-2}$&2.00\\
	\bottomrule
\end{tabular}
}
	\caption{Interpolation error of immersed $\mathcal{P}_3$ element of circular interface when when $\beta^+=1$, $\beta^-=50$}
	\label{tab:circle_itp_p3}
\end{table}

\begin{table}[htb]
\centering
\resizebox{.8\textwidth}{!}{
\begin{tabular}{ccccccc}
	\toprule
	$N$ & $||u_h-u||_{0}$ & order&$|u_h-u|_{1}$&order&$|u_h-u|_{2}$&order\\
	\midrule
	$10$&$2.9512\times 10^{-2}$&    &$1.6345\times 10^{-1}$&    &$4.2294\times 10^{0}$&\\
	$20$&$4.7377\times 10^{-3}$&2.64&$2.2606\times 10^{-2}$&2.85&$1.0358\times 10^{0}$&2.03\\
	$40$&$2.0030\times 10^{-4}$&4.56&$2.0387\times 10^{-3}$&3.47&$2.1701\times 10^{-1}$&2.25\\
	$60$&$4.3622\times 10^{-5}$&3.76&$5.0789\times 10^{-4}$&3.43&$8.4434\times 10^{-2}$&2.33\\
	$80$&$1.4458\times 10^{-5}$&3.84&$1.8888\times 10^{-4}$&3.44&$4.3573\times 10^{-2}$&2.30\\
	$100$&$6.0724\times 10^{-6}$&3.89&$8.8723\times 10^{-5}$&3.39&$2.6391\times 10^{-2}$&2.25\\
	\bottomrule
\end{tabular}
}
	\caption{Numerical solution errors of $\mathcal{P}_3$ element of circular interface when $\beta^+=1$, $\beta^-=50$}
	\label{tab:circle_sol_p3}
\end{table}

Finally, we examine the conditioning of the stiffness matrix associated with the scheme \ref{equ:scheme}. 
Figure  \ref{fig:cond} shows condition number versus $N$ for fixed $\beta^-=1$ and varying $\beta^+$. For the $\mathcal{P}_2$ space (left), the condition number scales like $\mathcal{O}(h^{-4})$; moreover, the magnitude is higher for large coefficient jumps than for moderate ones. Similar behavior is observed for the immersed $\mathcal{P}_3$ space (right). These results indicate regular conditioning of the proposed method. 

\begin{figure}[htb]
 \centering
 \includegraphics[width=.99\textwidth]{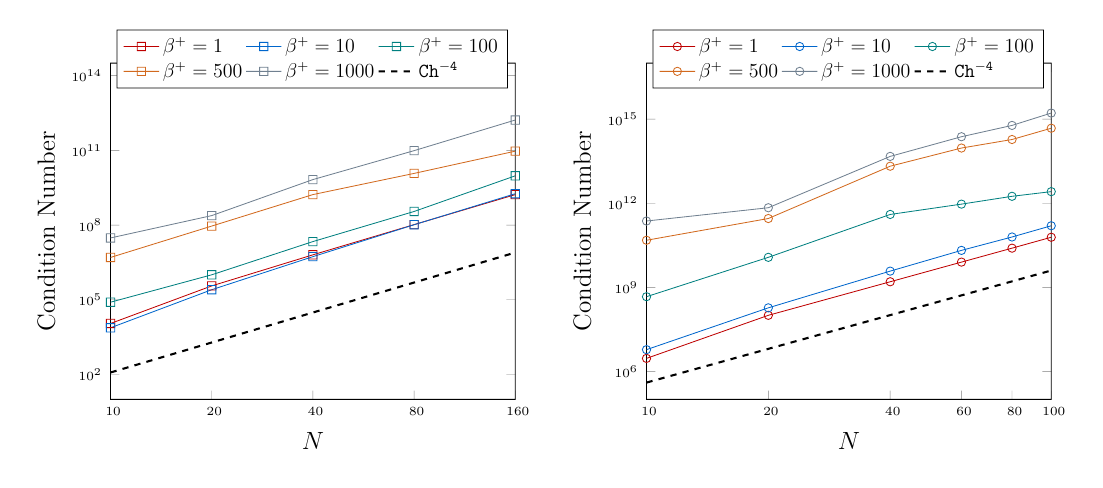}
 \caption{Condition number of $\mathcal{P}_2$ (left) and $\mathcal{P}_3$ (right) IPIFEM for circular interface case}
 \label{fig:cond}
\end{figure}

\subsection{A Flower Interface Case}
\label{sec:flower}
In this final example, we consider a relatively complex flower-shaped interface. The level-set function of the interface is defined as
\[\phi(x,y)=(x^2+y^2)[1+0.6\sin (6 \arctan (y/x))]-\left(1.5\pi/6.28\right)^4,\]
which splits the domain into $\Omega^+=\{(x,y)\in\Omega:\phi(x,y)>0\}$ and $\Omega^-=\{(x,y)\in\Omega:\phi(x,y)<0\}$. 
The shape of this interface is shown in the left plot of Figure \ref{fig:star}. We set the source term $f=0$ and homogenous boundary conditions. 

Since it is difficult to construct an exact solution for this geometry, we adopt a reference solution $\hat {u}$ computed on a very fine mesh ($N=320$) to evaluate errors. To facilitate high-order numerical quadrature on this irregular domain, we employ the technique described in 
\cite{2020CuiLengLiuZhangZheng,2015Saye}. The computed errors for the immersed $\mathcal{P}_2$ $C^0$ interior penalty method are reported in Table \ref{tab:flower_tb1}. The approximate error norms are defined as
\begin{equation}
	||u_h-\hat{u}||_{\tilde{0}} = h\left(\sum_{X\in\mathcal{N}_h}(u_h(X)-\hat{u}(X))^2\right)^{1/2},~~~
	|u_h-\hat{u}|_{\tilde{1}} = ||D_2 u_h-D_2\hat{u}||_{\tilde{0}},
\end{equation}
where $D_2 u$ denotes the second-order central difference approximation of $\nabla u$ on the grid function $u$. The surface plot of the numerical solution is shown in the right plot of Figure   \ref{fig:star}.

%
%

\begin{table}[htb]
\centering
\resizebox{.6\textwidth}{!}{
\begin{tabular}{ccccc}
	\toprule
	$N$ & $||u_h-\hat{u}||_{\tilde{0}}$ & order&$|u_h-\hat{u}|_{\tilde{1}}$&order\\
	\midrule
	$10$&$2.9510\times 10^{-3}$&    &$4.8539\times 10^{-3}$&    \\
	$20$&$1.2720\times 10^{-3}$&1.21&$2.0239\times 10^{-3}$&1.26\\
	$40$&$4.3351\times 10^{-4}$&1.55&$6.7660\times 10^{-4}$&1.58\\
	$80$&$1.0786\times 10^{-4}$&2.01&$1.6779\times 10^{-4}$&2.01\\
	$160$&$2.2620\times 10^{-5}$&2.25&$3.5572\times 10^{-5}$&2.24\\
	\bottomrule
\end{tabular}
}
	\caption{Numerical solution errors of $\mathcal{P}_2$ element of flower interface when $\beta^+=1$, $\beta^-=50$}
	\label{tab:flower_tb1}
\end{table}


\begin{figure}[htb]
 \centering
 \includegraphics[width=.46\textwidth]{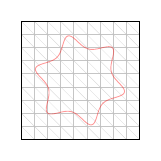}
 \includegraphics[width=.49\textwidth]{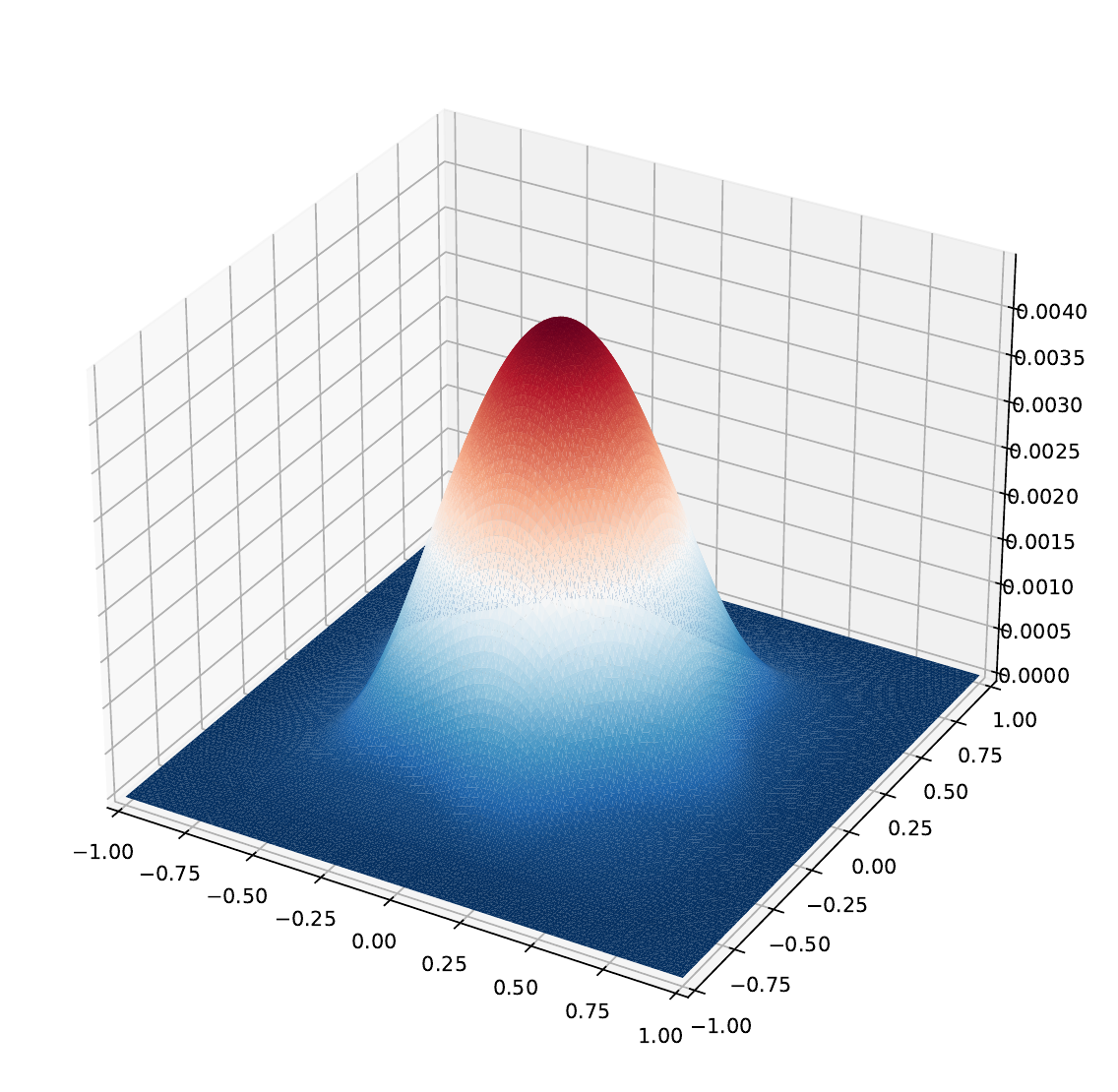}
 \caption{The shape (Left) and surface plot (Right) of the flower interface example}
 \label{fig:star}
\end{figure}

\section{Conclusion}
\label{sec:conclu}
In this paper, we developed a high-order immersed finite element framework for biharmonic interface problems, combining the least-squares enforcement of interface jump conditions with a modified $C^0$ interior penalty formulation. The proposed method achieves optimal convergence as demonstrated by numerical experiments with a range of interface geometries and coefficient contrasts.  Compared with existing approaches, it avoids the complexity of $C^1$ conforming elements and large computational costs of the fully discontinuous formulations, while maintaining flexibility on unfitted meshes. 

\bibliographystyle{plain}
\bibliography{a.bib}

 {
\appendix
\section{Details of Case 1 for Theorem \ref{thm:unisol}}
\label{Case1}
\underline{For $\mathcal{P}_2$ type I element}, we have the following scenarios
\begin{itemize}

	\item When $1/2\leq d<1$ and $1/2\leq e<1$, we have
	\begin{align}
		\det(\mathbf{M})=32 (d^2 + e^2)^2 (P_1+\rho P_2)\neq 0,	
	\end{align}
	where $d$ and $e$ are not equal to $1$ at the same time, and
	\begin{equation}
		\begin{split}
			P_1=&-d^3 (1 - 2 e)^2 e - 2 d^2 (-2 + e) e^3 - e^4 + d e^3 (-1 + 4 e) + d^4 (-1 + 4 e - 2 e^2) \geq P_1(1/2,1/2) \geq 1/8,\\
			P_2=&2 d^4 (-1 + e)^2 + d^3 (1 - 2 e)^2 e + 2 d^2 (-1 + e)^2 e^2 + 2 e^4 +  d (e^3 - 4 e^4) \geq P_2(1,1) \geq 0.
		\end{split}
	\end{equation}

	\item When $0< d\leq 1/2$ and $0<e\leq 1/2$, we have
\begin{align}
	\det(\mathbf{M})=-32 (d^2 + e^2)^2(-2 d^2 e^2 (d + e)^2+\rho P_1)\neq 0,
\end{align}
where
\begin{equation}
\begin{split}
	P_1=4 d^3 e^3 - e^4 + 2 d^2 e^2 (-1 + e^2) + d^4 (-1 + 2 e^2) \leq P_1(0,0) \leq 0.
\end{split}
\end{equation}

	\item When $0< d\leq 1/2$ and $1/2\leq e<1$ (the case when $1/2\leq d<1$ and $0< e\leq 1/2$ is similar), we have
\begin{align}
	\det(\mathbf{M})=-32 (d^2 + e^2)^2(P_1+\rho P_2)\neq 0,
\end{align}
where
\begin{equation}
\begin{split}
	P_1=&-d^2 (-2 e^4 + d (e - 4 e^2) + d^2 (1 - 4 e + 2 e^2)) \geq P_1(0,3/4) \geq 0,\\
	P_2=&2 d^4 (-1 + e)^2 + e^4 + d^3 (e - 4 e^2) - 2 d^2 e^2 (-1 + e^2) \geq P_2(0,1/2) \geq 1/16.
\end{split}
\end{equation}

\end{itemize}
\underline{For $\mathcal{P}_2$ type II element},
\begin{itemize}

	\item  When $1/2\leq d<1$ and $1/2\leq e<1$, we have
\begin{align}
	\det(\mathbf{M})=32 ((e-1)^2+(d-e)^2)^2(P_1+\rho P_2) \neq 0,	
\end{align}
where 
\begin{equation}
\begin{split}
	P_1=&-1 + d^4 + 4 e - 4 d^3 e - 4 e^2 + 2 e^4 - 2 d^2 e (-2 + 2 e - 4 e^2 + e^3)\\
	    &+  4 d (1 - 5 e + 8 e^2 - 6 e^3 + e^4) \geq P_1(1,1) \geq 0,\\
	P_2=& 2 (-1 + d)^2 (-1 + e)^4 \geq 0.
\end{split}
\end{equation}

	\item  When $0< d\leq 1/2$ and $0<e\leq 1/2$, we have
\begin{align}
	\det(\mathbf{M})=32 ((e-1)^2+(d-e)^2)^2(P_1+\rho P_2) \neq 0
\end{align}
where
\begin{equation}
\begin{split}
	P_1=&d^4 + e^3 + d^3 (3 - 11 e + 4 e^2) + d e^2 (1 - 9 e + 4 e^2) + 
 d^2 (2 - 13 e + 31 e^2 - 16 e^3 + 2 e^4) \geq P_1(0,0) \geq 0, \\
	P_2=&(1 - e) (1 - 3 e + 5 e^2 - 4 e^3 + d^3 (-3 + 4 e) + 
     d^2 e (9 - 14 e + 2 e^2) + d e (-4 + 3 e + 4 e^2)) \\
     \geq & P_2(1/2,1/2) \geq 1/32.
\end{split}
\end{equation}

	\item When $0< d\leq 1/2$ and $1/2\leq e<1$, we have
\begin{align}
	\det(\mathbf{M})=-32 ((e-1)^2+(d-e)^2)^2(P_1+\rho P_2) \neq 0
\end{align}
where
\begin{equation}
\begin{split}
	P_1=&d^4 - 4 d^3 (-1 + 4 e - 3 e^2 + e^3) + 2 d^2 e (-2 + 10 e - 8 e^2 + 3 e^3) + e (-1 + 5 e - 7 e^2 + 4 e^3) \\
	    &- d (-1 + 3 e + e^2 - 3 e^3 + 4 e^4) \geq P_1(1/2,0.696565) \geq  0.0128466, \\
	P_2=&(-1 + e)^3 (-1 + d + 4 d^3 + 4 d e - 2 d^2 (1 + 3 e)) \geq P_2(0,1) \geq 0.
\end{split}
\end{equation}

	\item When $1/2\leq d<1$ and $0< e\leq 1/2$, we have
\begin{align}
	\det(\mathbf{M})=-32 ((e-1)^2+(d-e)^2)^2(P_1+\rho P_2) \neq 0,
\end{align}
where
\begin{equation}
\begin{split}
	P_1=&-1 + 3 d + 2 d^2 - d^3 + d^4 + 2 e - 17 d e - 5 d^2 e + d^3 e - 2 e^2 + 34 d e^2 + 7 d^2 e^2 - 8 d^3 e^2 - 36 d e^3  \\
	    &+ 8 d^2 e^3 + 4 d^3 e^3 + 12 d e^4 - 6 d^2 e^4 + 3 e (1 - 2 e + 2 e^2) - e^2 (1 - 2 e + 2 e^2) \geq P_1(1/2,1/2) \geq  1/32, \\
	P_2=&-(-1 + d)^2 (-1 + e) (2 + d (1 - 2 e)^2 - 7 e + 10 e^2 - 6 e^3) \geq P_2(1,1/4) \geq 0.
\end{split}
\end{equation}

\end{itemize}
\underline{For $\mathcal{P}_3$ type I element},
\begin{itemize}

	\item When $0\leq d<1/3$ and $0\leq e<1/3$, we have
\begin{align}
	\det(\mathbf{M})=-612220032 (d^2 + e^2)^6 \rho ^2(P_1+\rho P_2)\neq 0,	
\end{align}
where $d$ and $e$ are not equal to $1$ at the same time and 
\begin{equation}
\begin{split}
	P_1=&9 d^2 e^2 (-d^6 (-1 + e) + d^5 (2 - 3 e) e + e^4 (-2 + e^2) + d^4 (-2 + e + 3 e^2 - 4 e^3) \\
		&+ d^3 e (-4 + 7 e + 4 e^2 - 4 e^3) + d^2 e^2 (-4 + 7 e + 3 e^2 - 3 e^3) + d e^3 (-4 + e + 2 e^2 - e^3))\\
		 \leq & P_1(0,1/6) \leq 0,\\
	P_2=&9 d^7 e^3 (-2 + 3 e) + e^6 (-2 + e^2) + 9 d^3 e^5 (4 - e - 2 e^2 + e^3) + 9 d^5 e^3 (4 - 7 e - 4 e^2 + 4 e^3) \\
	    &+ d^8 (1 - 9 e^2 + 9 e^3) + 3 d^4 e^2 (-2 + 14 e^2 - 21 e^3 - 9 e^4 + 9 e^5) + d^6 (-2 + 22 e^2 - 9 e^3 - 27 e^4 + 36 e^5) \\
	    &+ d^2 (-6 e^4 + 22 e^6 - 9 e^8) \leq P_1(0,0) \leq 0.
\end{split}
\end{equation}

	\item When $0\leq d<1/3$ and $1/3\leq e<2/3$, we have
\begin{align}
	\det(\mathbf{M})=51018336 (d^2 + e^2)^4 \rho  (P_1+\rho P_2+\rho ^2P_3) \neq 0
\end{align}
where $P_1\leq P_1(0,1/2)\leq 0$, $P_2\leq P_2(0,1/2)\leq 0$, $P_3 \leq P_3(0,1/3)\leq -68/177147$.

	\item When $0\leq d<1/3$ and $2/3\leq e<1$, we have
\begin{align}
	\det(\mathbf{M})=-51018336 (d^2 + e^2)^4 \rho  (P_1+\rho P_2+\rho ^2P_3) \neq 0
\end{align}
where $P_1\leq P_1(0,5/6)\leq 0$, $P_2\leq P_2(0,5/6)\leq 0$, $P_3 \leq P_3(1/3,2/3)\leq -29240/177147$.

	\item When $1/3\leq d<2/3$ and $1/3\leq e<2/3$, we have
\begin{align}
	\det(\mathbf{M})=25509168 (d^2 + e^2)^3 (P_1+\rho P_2+\rho ^2P_3+\rho ^3P_4) \neq 0
\end{align}
where $P_1\geq P_1(1/3,1/2)=0$, $P_2\geq P_2(1/3,1/3)=0$, $P_3\geq P_3(1/3,1/3)=16384/4782969$, $P_4\geq P_4(2/3,2/3)=0$.

	\item When $1/3\leq d<2/3$ and $2/3\leq e<1$ and $O$ belongs to $T^+$, we have
\begin{align}
	\det(\mathbf{M})=-25509168 ((d^2 + e^2)^3) (P_1+\rho P_2+\rho ^2P_3+\rho ^3P_4) \neq 0
\end{align}
where $P_1$, $P_2$, $P_3$ and $P_4$ are no less than $0$.

	\item When $1/3\leq d<2/3$ and $2/3\leq e<1$ and $O$ belongs to $T^-$, we have
\begin{align}
	\det(\mathbf{M})=-25509168 (d^2 + e^2)^3 (P_1+\rho P_2+\rho ^2P_3+\rho ^3P_4) \neq 0
\end{align}
where $P_1$, $P_2$, $P_3$ and $P_4$ are no more than $0$.

	\item When $2/3\leq d<1$ and $2/3\leq e<1$, we have
\begin{align}
	\det(\mathbf{M})=102036672 (d^2 + e^2)^4 (P_1+\rho P_2+\rho ^2P_3+\rho ^3P_4) \neq 0
\end{align}
where $P_1$, $P_2$, $P_3$, and $P_4$ are no more than $0$.

\end{itemize}
\underline{For $\mathcal{P}_3$ type II element},

\begin{itemize}

	\item When $2/3\leq d<1$ and $2/3\leq e<1$, we have
\begin{align}
	\det(\mathbf{M})=-612220032 ((d-e)^2+(e-1)^2)^6(P_1+\rho P_2) \neq 0
\end{align}
where $P_1$ and $P_2$ are no more than $0$.

	\item When $2/3\leq d<1$ and $1/3\leq e<2/3$, we have
\begin{align}
	\det(\mathbf{M})=51018336 ((d-e)^2+(e-1)^2)^4(P_1+\rho P_2+\rho ^2P_3) \neq 0
\end{align}
where $P_1$, $P_2$, and $P_3$ are no more than $0$.

	\item When $2/3\leq d<1$ and $0\leq e<1/3$, we have
\begin{align}
	\det(\mathbf{M})=-51018336 ((d-e)^2+(e-1)^2)^4(P_1+\rho P_2+\rho ^2P_3) \neq 0
\end{align}
where $P_1$, $P_2$, and $P_3$ are no more than $0$.

	\item When $1/3\leq d<2/3$ and $1/3\leq e<2/3$, we have
\begin{align}
	\det(\mathbf{M})=-25509168 ((d-e)^2+(e-1)^2)^3(P_1+\rho P_2+\rho ^2P_3) \neq 0
\end{align}
where $P_1$, $P_2$, and $P_3$ are no more than $0$.

	\item When $1/3\leq d<2/3$ and $0\leq e<1/3$ and $O$ belongs to $T^-$, we have
\begin{align}
	\det(\mathbf{M})=25509168 ((d-e)^2+(e-1)^2)^3(P_1+\rho P_2+\rho ^2P_3+\rho ^3P_4) \neq 0
\end{align}
where $P_1$, $P_2$, $P_3$, and $P_4$ are no more than $0$.

	\item When $1/3\leq d<2/3$ and $0\leq e<1/3$ and $O$ belongs to $T^+$, we have
\begin{align}
	\det(\mathbf{M})=25509168 ((d-e)^2+(e-1)^2)^3(P_1+\rho P_2+\rho ^2P_3+\rho ^3P_4) \neq 0
\end{align}
where $P_1$, $P_2$, $P_3$, and $P_4$ are no less than $0$.

	\item When $0\leq d<1/3$ and $0\leq e<1/3$, we have
\begin{align}
	\det(\mathbf{M})=102036672 ((d-e)^2+(e-1)^2)^4\rho(P_1+\rho P_2+\rho ^2P_3) \neq 0
\end{align}
where $P_1$, $P_2$, and $P_3$ are no more  than $0$.

\end{itemize}

}

\end{document}